\documentclass[a4paper,11pt]{article}

\usepackage{amsmath,amssymb,epsfig
}
\usepackage{subfigure}
\usepackage{paralist}

\newcommand{\bR}{{\bf R}}

\newcommand{\be}{ \tilde{e}}

\newcommand{\btx}{ \tilde{x}}

\newcommand{\tv}{ \tilde{v}}
\newcommand{\tn}{ \tilde{n}}

\newcommand{\bty}{ \tilde{y}}
\newcommand{\btz}{ \tilde{z}}

\newcommand{\diag}{\operatorname{diag}}

\newcommand{\co}{\operatorname{co}}

\newcommand{\n}{\nonumber}

\newcommand{\cC}{{\mathcal C}}
\newcommand{\cD}{{\mathcal D}}

\newcommand{\cJ}{{\mathcal J}}

\newcommand{\cN}{{\mathcal N}}

\newcommand{\cT}{{\mathcal T}}

\newcommand{\T}{{\mathfrak S}}

\makeatletter
\@addtoreset{equation}{section}

\makeatother

\newtheorem{theorem}{Theorem}[section]
\newtheorem{definition}[theorem]{Definition}
\newtheorem{lemma}[theorem]{Lemma}
\newtheorem{remark}[theorem]{Remark}

\newenvironment{proof}{\par \vspace{0.3cm} \noindent{\sc Proof:} \ignorespaces}%
{\nolinebreak\hfill $\square$\par \medskip}

\newcommand{\diam}{\mathop{\mathrm{diam}}\limits}

\def\R{ \mathbb {R}}

\newcommand{\N}{\mathbb{N}}
\def\cN{\mathcal{N}}

\newcommand{\tx}{\tilde{x}}

\newcommand{\ty}{\tilde{y}}

\newcommand{\tf}{\tilde{f}}

\begin{document}

\sloppy

\newcounter{fig}

\title{Construction of a CPA contraction metric for\\ periodic orbits using semidefinite optimization\thanks{This work  was supported by the Engineering and Physical Sciences Research Council [grant number
EP/J014532/1].}}

\author{Peter Giesl\thanks{
email \texttt{p.a.giesl@sussex.ac.uk}}\\
         Department of Mathematics\\
         University of Sussex\\
         Falmer BN1 9QH\\
         United Kingdom
         \and
Sigurdur Hafstein\thanks{
email \texttt{sigurdurh@ru.is}}\\
School of Science and Engineering\\
  Reykjavik University\\
Menntavegi 1\\
IS-101 Reykjavik\\
Iceland}

\date{\today}

\maketitle

\begin{abstract}
A Riemannian metric with a local contraction property can be used to prove existence and uniqueness of a periodic orbit and determine a subset of its basin of attraction. While the existence of such a contraction metric is equivalent to the existence of an exponentially stable periodic orbit, the explicit construction of the metric is a difficult problem.

In this paper, the construction of such a contraction metric is achieved by formulating it as an equivalent problem, namely a feasibility problem in semidefinite optimization. The contraction metric, a matrix-valued function, is constructed as a continuous piecewise affine (CPA) function, which is affine on each simplex of a triangulation of the phase space. The contraction conditions are formulated as conditions on the values at the vertices.

The paper states a semidefinite optimization problem. We prove on the one hand that a feasible solution of the optimization problem determines a CPA contraction metric and on the other hand that the optimization problem is always feasible if the system has an exponentially stable periodic orbit and the triangulation is fine enough. An objective function can be used to obtain a bound on the largest Floquet exponent of the periodic orbit.
 \end{abstract}

\section{Introduction}

%

In this paper we consider a time-periodic ODE of the form $\dot{x}=f(t,x)$, where $f(t,x)=f(t+T,x)$ for all $(t,x)\in \mathbb R\times \mathbb R^n$ with a fixed period $T>0$, and study the basin of attraction of a periodic solution. 

The basin of attraction can be computed using a variety of methods:
{\bf Invariant manifolds} form the boundaries of basins of attraction, and their computation can thus be used to find a basin of attraction \cite{krauskopf2,b-wiggins}. However, this method needs additional arguments to ensure that a certain region is the basin of attraction of an attractor, and that, for example, there are no other attractors in that region. Other approaches to compute the basin of attraction are for example the {\bf cell mapping approach} \cite{Hsu} or {\bf set oriented
methods} \cite{Dellnitz_Junge} which divide the phase space into cells and compute
the dynamics between these cells.


{\bf Lyapunov functions} \cite{lyap}  are a natural way of analysing the basin of attraction, since they start from the attractive solution, not from the boundary. Moreover,  through their level sets, they  give additional
information about the basin of attraction than just the
boundary.
%
{\bf Converse theorems} which guarantee the
 existence of a Lyapunov
function under certain conditions have been given by many authors,
for an overview see \cite{hahn}.
However, all converse theorems offer no general
method to analytically construct Lyapunov functions.

Recently, several methods to construct Lyapunov functions have been proposed:
%
Hafstein  \cite{MON}  constructed a piecewise affine Lyapunov function using {\bf linear programming}.
Parrilo \cite{parrilo} and Papachristodoulou and Prajna in \cite{papac}
consider the numerical construction of Lyapunov functions that are
presentable as sum of squares of polynomials (\textbf{SOS}) for autonomous polynomial systems.  These ideas have been taken further by recent
publications of Peet \cite{Peet2009} and Peet and Papachristodoulou \cite{Peet2010}, where the existence of a polynomial
Lyapunov function on bounded regions for exponentially stable systems in proven.

A different method deals with {\bf Zubov's equation} and computes a
solution of this partial differential equation (PDE). In Camilli et al. \cite{camilli01}, Zubov's method was extended to
control problems.
Giesl  considered a particular Lyapunov function satisfying
 a linear PDE which was solved using {\bf meshless collocation},  in particular Radial Basis Functions
\cite{Giesl-07}.
This method has been extended to time-periodic ODEs \cite{Giesl-Wendland-09-1}.

Lyapunov functions attain their minimum on the attractor and have a negative orbital derivative for all points in the basin of attraction apart from the attractor. Hence, it is necessary to have {\bf exact information} about the attractor in the phase space before one can compute a Lyapunov function. Whereas this information might be easy to obtain in special examples, in general this information is not available.

\vspace{0.3cm}

\noindent
{\bf Local contraction property -- Borg's criterion}

Another method to characterise the basin of attraction, introduced by
Borg \cite{borg}, uses a local contraction property and does {\bf not} require information about the periodic orbit.
 Let $M(t,x)$ be a Riemannian metric, i.e. $M\in C^1(\mathbb R\times \mathbb R^n,
\mathbb R^{n\times n})$ such that $M(t,x)$ is a positive definite, symmetric $(n\times n)$ matrix for all $(t,x)$.
  Then $\langle v,w\rangle_{M(\tx)}:=
v^TM(\tx)w$ defines a point-dependent scalar product, where $v,w\in \mathbb R^n$. The sign of the real-valued function $\widetilde{L}_M(t,x)$, cf. (\ref{LdefM1}),
then describes whether the solution through $(t,x)$ and adjacent solutions approach each other with respect to
the Riemannian metric $M$. We define
\begin{eqnarray}
\widetilde{ L}_M(t,x)&:=&\max_{w\in \mathbb R^n,w^TM(t,x)w=1}w^T\left[M(t,x)D_xf(t,x)
+\frac{1}{2}M'(t,x)\right]w,\label{LdefM1}
\end{eqnarray}
where $ M'(t,x)$ denotes the orbital derivative
of $M(t,x)$, which is the derivative along solutions of $\dot{x}=f(t,x)$.

If $\widetilde{L}_M(t,x)$ is negative for all $(t,x)\in K$
where $K$ is a positively invariant, connected set, then $K$ is a subset of the
basin of attraction of a unique periodic orbit in $K$.

The maximum in (\ref{LdefM1}) is taken over all $w\in \mathbb R^n$ with a norm condition, and $\widetilde{L}_M(t,x)<0$ is equivalent to $L_M(t,x)<0$ given by (\ref{cont}), where $\lambda_{max}(\cdot)$  denotes the maximal eigenvalue of a symmetric matrix. Here, we use that $w^T M(t,x) D_xf(t,x)w=w^T D_xf(t,x)^TM(t,x)w$.
\begin{eqnarray}
{ L}_M(t,x)&:=&\max_{w\in \mathbb R^n,\|w\|=1}w^T\big[M(t,x)D_xf(t,x)
+D_xf(t,x)^TM(t,x)
+M'(t,x)\big]w\nonumber\\
&=&\lambda_{max}\left(M(t,x)D_xf(t,x)
+D_xf(t,x)^TM(t,x)
+M'(t,x)\right).\label{cont}
\end{eqnarray}

We seek to find a matrix-valued function $M$ satisfying $L_M(t,x)<0$. This is equivalent to the condition that the symmetric matrix
$$M(t,x)D_xf(t,x)
+D_xf(t,x)^TM(t,x)
+M'(t,x)$$ is negative definite. As this is a Linear Matrix Inequality, it can be formulated as a constraint of a semidefinite optimization problem.

While the sufficiency of this local contraction criterion in the autonomous case goes back to \cite{borg,hartman/olech,stenstroem,leonov96}, its necessity was shown in \cite{Giesl-04-1}. The method was extended to time-periodic systems \cite{zaa}.

The advantage of this method over, for example, Lyapunov functions, is that it does not require information about the position of the periodic orbit. Moreover, the criterion is robust to small errors.


Although the existence of Riemannian metrics has been shown \cite{Giesl-04-1}, it remains a difficult problem to construct them for concrete examples. This is a similar problem to the construction of a (scalar-valued) Lyapunov function, but Borg's criterion requires the construction of a matrix-valued function $M(t,x)$. In the two-dimensional autonomous case, however, there exists a special Riemannian metric of the form $M(x)=e^{2W(x)}I$, where $W$ is a scalar-valued function  \cite{Giesl-04-1}. This can be
used to find an approximation using Radial Basis Functions
\cite{Giesl-07-5}.
In higher dimensions, however, the existence of such a special Riemannian metric is not true in general  \cite{Giesl-04-1}.  In  \cite{Giesl-09-1}, a combination of a Riemannian metric locally near the periodic orbit with a Lyapunov function further away was used, and the construction was again achieved by Radial Basis Functions. This method, however, heavily depends on information about the periodic orbit, which was obtained  by a numerical approximation of the periodic orbit and its variational equation.

In this paper, we will develop a new method to construct a Riemannian metric to fulfill Borg's criterion, which will use semidefinite optimization and does not require any information about the periodic orbit.

\vspace{0.3cm}
\noindent
{\bf Semidefinite Optimization}

A semidefinite optimization problem for the variables $y_1,\ldots,y_m$ is of the form
$$
\begin{array}{ll}
\mbox{minimize}& \sum_{i=1}^m c_iy_i\\
\mbox{subject to}& \sum_{i=1}^m F_i y_i -F_0=X\succeq 0,\end{array}
$$
where $F_i$ are symmetric $(N\times N)$ matrices and $X\succeq 0$ means that the matrix $X$ is positive semidefinite.
%

The goal of this paper is to formulate the condition of a contraction metric as a semidefinite optimization (feasibility) problem. In a subsequent paper we will discuss the details of how to solve this problem efficiently.

The main idea is to first triangulate the phase space. The Riemannian metric, i.e. the symmetric matrix $M(t,x)$, will be expressed as a continuous piecewise affine (CPA) function, i.e. if $M$ is given at the vertices  $(t_0,x_0),\ldots,(t_{n+1},x_{n+1})$ of a simplex, then $M(t,x)=\sum_{i=0}^{n+1}\lambda_i M(t_i,x_i)$, where
$(t,x)=\sum_{i=0}^{n+1}\lambda_i (t_i,x_i)$. 
 The conditions of Borg's criterion will become the constraints of a semidefinite optimization problem on the vertices of the triangulation, which will ensure the contraction property for {\bf all} points in the simplices.

In \cite{Aylwarda}, a contraction metric is also constructed using semidefinite optimization. There are, however, three main differences to our approach: firstly, adjacent trajectories in {\bf all} directions are contracted, whereas in our case the contraction takes place in the $n$-dimensional subspace $\mathbb R^n$ of $\mathbb R^{n+1}$, but not in the time-direction. Thus, in our case, the attractor is a periodic orbit, whereas in their case, it is an equilibrium point. Secondly, and more importantly, the above paper transforms the construction problem to a Linear Matrix Inequality and solves this using a sum-of-squares approach. This approach is used to prove global stability, i.e. the basin of attraction is the whole space. The contraction metric is a polynomial function, and the system considered is assumed to be polynomial, too. In this paper, we study systems which are not necessarily polynomial nor globally stable, and we triangulate the phase space to obtain a large subset of the basin of attraction. Lastly, we are able to prove that the semidefinite optimization problem is feasible if and only if the dynamical system has an exponentially stable periodic orbit, whereas in their paper the equivalence does not hold since the sum-of-squares condition is more restrictive than positive definiteness of matrices.

The paper is organised as follows: In Section \ref{sec0} we start with preliminaries, and in Section \ref{sec1} we generalise the existing theorem for a smooth Riemannian metric $M$ to a CPA (continuous piecewise affine) Riemannian contraction metric. We show that the existence of such a metric is sufficient to prove existence and uniqueness of a periodic orbit and to determine a subset of its basin of attraction. In Section \ref{sec2}, we describe the triangulation of the phase space into a simplicial complex and state the semidefinite optimization problem. Furthermore, we show that the feasibility of the semidefinite optimization problem provides us with a CPA contraction metric. We also discuss a possible objective function to obtain a bound on the largest Floquet exponent. In Section \ref{sec3}, finally, we show that the  semidefinite problem is feasible, if the dynamical system has an exponentially stable periodic orbit.

Altogether, this paper shows that the problem of finding a contraction metric is equivalent to the feasibility of a semidefinite optimization problem.

\section{Preliminaries}\label{sec0}

In this paper we consider a time-periodic ODE of the form
\begin{eqnarray}
\dot{x}&=&f(t,x),\label{ODE}
\end{eqnarray} where $f(t,x)=f(t+T,x)$ for all $(t,x)\in \mathbb R\times \mathbb R^n$ with a fixed period $T>0$. We denote $\tf(t,x)=\left(\begin{array}{c}1\\f(t,x)\end{array}\right)$ and $\tx=(t,x)$.
We study the equation on the cylinder $S_T^1\times \mathbb R^n$ as the phase space, where $S_T^1$ denotes the circle of circumference $T$. We assume that $f\in C^2(S^1_T\times \mathbb R^n,\mathbb R^n)$. If even $f\in C^3(S^1_T\times \mathbb R^n,\mathbb R^n)$ holds, then we obtain improved results, in particular higher order approximations; this is, however, not necessary to derive the main result. We denote the (unique) solution of the ODE with initial value $x(t_0)=x_0$ by
$x(t)=:S_t^x(t_0,x_0)$ and denote $(t+t_0\mbox{ mod }T,x(t))=:S_t(t_0,x_0)\in S_T^1\times \mathbb R^n$. Furthermore,  we assume that the solution exists for all $t\ge 0$.

We use the usual notations for the vector and matrix norms, in particular we denote by $\|M\|_{max}:=\max_{i,j=1,\ldots,n}|M_{ij}|$ the maximal entry of a matrix. $\mathbb S^n$ denotes the set of all symmetric real-valued $n\times n$ matrices. For a symmetric matrix $M\in \mathbb S^n$, $\lambda_{max}(M)$ denotes its maximal eigenvalue, we write $M\succeq 0$ if and only if $M$ is positive semidefinite and  $M\preceq 0$ if and only if $M$ is negative semidefinite.  The convex hull is defined by
$$\co(\tx_0,\tx_1,\ldots,\tx_{n+1}):=\left\{\sum_{i=0}^{n+1}\lambda_i\tx_i\colon \lambda_i\ge 0,\sum_{i=0}^{n+1}\lambda_i=1\right\}.$$

\section{A CPA contraction metric is sufficient for a periodic orbit}
\label{sec1}

It was shown in \cite{zaa} that a smooth contraction metric implies the existence and uniqueness of a periodic orbit and gives information about its basin of attraction. In this paper, we will seek to construct a CPA contraction metric, which is not of the smoothness required in the above paper.

In this section we show that we can relax the conditions on the smoothness of $M$ to cover the case of a CPA contraction metric. We will require that $M$ is continuous, Lipschitz continuous with respect to $x$, and  the forward orbital derivative exists. We will later relate this to the construction of the CPA contraction metric on a suitable triangulation, and we will prove that such a CPA metric will satisfy all assumption that we make in this section (Lemma \ref{le4.9}). The proof of the main theorem, Theorem \ref{th1}, will closely follow \cite{zaa}.

First we define a weaker notion of a Riemannian metric, which does not assume that $M$ is differentiable, but only that the orbital derivative in forward time exists.

\begin{definition}[Riemannian metric]\label{Riem}
 $M$ is called a Riemannian metric for (\ref{ODE}), if $M\in C^0(S^1_T \times \mathbb R^n,\mathbb S^n)$ where $M(t,x)$ is positive definite.
 Moreover, we assume that $M$ is locally Lipschitz-continuous with respect to $x$, i.e. for all $(t_0,x_0)\in S_T^1\times \mathbb R^{n}$ there exists a neighborhood $U\subset S_T^1\times \mathbb R^{n}$ such that
 $$\|f(t,x)-f(t,y)\|\le L \|x-y\|$$ holds for all $(t,x),(t,y)\in U$.
Furthermore, we assume that the forward orbital derivative
$M'_+(t,x)$ is defined for all $(t,x)\in S^1_T\times \mathbb R^n$, where
 $M'_+(t,x)$ denotes the matrix
$$M'_+(t,x)=\lim_{\theta\to 0^+} \frac{M(S_\theta (t,x))-M(t,x)}{\theta}.$$
\end{definition}

We have the following statements for functions with a right-hand side derivative.

\begin{lemma}\label{product_rule}
Let $g_1,g_2 \in C^0(\mathbb R,\mathbb R)$ be RHS differentiable, i.e. $g'_+(x):=\lim_{h\to 0^+}\frac{g(x+h)-g(x)}{h}$ exists. Let $G\colon \mathbb R\to\mathbb R$ be differentiable.

Then $g_1(x)\cdot g_2(x)$ and $G(g(x))$ are RHS differentiable and
\begin{eqnarray}
(g_1\cdot g_2)'_+(x)&=&(g_1)'_+(x)\cdot g_2(x)+g_1(x)\cdot (g_2)'_+(x)\\
(G\circ g)'_+(x)&=&G'(g(x))\cdot g'_+(x).
\end{eqnarray}

Let $g \in C^0(I,\mathbb R)$, where $I\subset \mathbb R$ is open, be RHS differentiable.
Let $[ x_1, x_2]\subset I$, and let $\int_{x_1}^{x_2} g'_+(\xi)\,d\xi$ exist and be finite.

Then
 we have
$$\int_{x_1}^{x_2} g'_+(\xi)\,d\xi=g(x_2)-g(x_1).$$
\end{lemma}
\begin{proof}
The first two statements follow directly from the usual proofs, replacing the limit with the RHS limit.
The last statement is a result due to Lebesgue, formulated originally for Dini derivatives.
\end{proof}

\begin{lemma}\label{ableitung}

Let $V(t,x)$ be a function which is locally Lipschitz with respect to $x$  and let $\lim_{\theta \to 0^+}\frac{V( (t_0,x_0)+\theta \tf(t_0,x_0))-V(t_0,x_0)}{\theta}$ exist.

Then the orbital derivative $V'_+(t_0,x_0)$ exists and is equal to
\begin{eqnarray*}
V'_+(t_0,x_0)&:=&\lim_{\theta \to 0^+}\frac{V(S_\theta (t_0,x_0))-V(t_0,x_0)}{\theta}\\
&=&\lim_{\theta \to 0^+}\frac{V( (t_0,x_0)+\theta\tf(t_0,x_0))-V(t_0,x_0)}{\theta}.
\end{eqnarray*}
\end{lemma}
\begin{proof}
Recall that $\tf(t,x)=\left(\begin{array}{c}1\\f(t,x)\end{array}\right)$ and $\tx=(t,x)$.
We use that
\begin{eqnarray*}
\frac{V(S_\theta (t_0,x_0))-V(t_0,x_0)}{\theta}&=&
\frac{V(S_\theta (t_0,x_0))-V( (t_0,x_0)+\theta (1,f(t_0,x_0)))}{\theta}\\
&&+\frac{V( (t_0,x_0)+\theta(1,f(t_0,x_0)))-V(t_0,x_0)}{\theta}\\
\end{eqnarray*}
Note that due to the Lipschitz continuity $|V(t,y)-V(t,x)|\le  L\|y-x\|$ we have for small $\theta$
\begin{eqnarray*}
\left|\frac{V(S_\theta (t_0,x_0))-V( (t_0,x_0)+\theta(1,f(t_0,x_0)))}{\theta}\right|
&\le&L
\frac{\|S_\theta^x (t_0,x_0)- x_0-\theta f(t_0,x_0)\|}{\theta}.
\end{eqnarray*}
By the Mean Value Theorem there are $h_1(\theta),\ldots,h_n(\theta)\in [0,1]$ such that
$$(S_\theta^x(t_0,x_0)-x_0)_i=\theta
\frac{\partial}{\partial t}(S_{h_i (\theta)\theta  }^x(t_0,x_0))_i=\theta
f_i(S_{h_i(\theta) \theta}(t_0,x_0)).$$
Hence,
\begin{eqnarray*}
\lim_{\theta\to 0^+}
\frac{|(S_\theta^x (t_0,x_0)- x_0-\theta f(t_0,x_0))_i|}{\theta}
&\le&\lim_{\theta \to 0^+}\left|f_i(S_{h_i (\theta)\theta}(t_0,x_0))-f_i(t_0,x_0)\right|\\
&=&0
\end{eqnarray*}
since the solution and $f_i$ are both continuous. Altogether, we thus have
\begin{eqnarray*}
\lim_{\theta \to 0^+}\left|\frac{V(S_\theta (t_0,x_0))-V( (t_0,x_0)-\theta(1,f(t_0,x_0)))}{\theta}\right|
&=&0.
\end{eqnarray*}
\end{proof}

\begin{theorem}\label{th1}
Consider the equation $\dot{x}=f(t,x)$,
 where $x\in \mathbb R^n$, and assume that $f\in C^0(S^1_T \times
\mathbb R^n,\mathbb R^n)$
and all partial derivatives of $f$ order one with respect to $x$ are
continuous functions of $(t,x)$.
Let $\varnothing\not= G\subset  S^1_T\times\mathbb R^n$
 be a connected, compact and positively invariant
set. Let $M$ be a Riemannian metric in the sense of Definition \ref{Riem}.

 Moreover, assume ${L}_M(t,x)\le -\nu <0$ for all $(t,x)\in G$, where
\begin{eqnarray}
{ L}_M(t,x)&:=&\sup_{w\in \mathbb R^n,w^TM(t,x)w=1}w^T\left[M(t,x)D_xf(t,x)
+\frac{1}{2}M'_+(t,x)\right]w\label{LM}
\end{eqnarray}

Then there exists one and only one periodic orbit  $\Omega\subset G$ which
is exponentially asymptotically stable. Moreover, for its basin
of attraction $G\subset A(\Omega)$ holds.

If $$\int_0^T p(t)^T M'_+(t,x(t))p(t)\, dt$$ exists and is finite for all solutions
$x(t)$ with $x(0)\in G$ and all functions $p\in C^0([0,T], \mathbb R^n)$, then
  the largest real part $-\nu_0$
of all Floquet exponents of
$\Omega$ satisfies \begin{eqnarray*}-\nu_0\le -\nu.
\end{eqnarray*}
\end{theorem}
Note that in contrast to \cite{zaa}, here $L_M(t,x)$ is not necessarily continuous, since $M'_+(t,x)$ is not continuous in general.
\begin{proof}
The only parts in the proof that need to be changed slightly from \cite{zaa} are the proof of Proposition 3.1 and the estimate on the Floquet exponents; the other steps of the proof are exactly the same.

\noindent {\bf Proposition 3.1}

We replace the temporal derivative in Step III of \cite[Proposition 3.1]{zaa} by the {\it  forward} temporal derivative of
 \begin{eqnarray*}
A(\theta)^2&:=&
\left[S^x_{\theta}(t_0,x_0+\eta)
-S^x_\theta(t_0, x_0)\right]^T
M(S_\theta(t_0, x_0))\nonumber\\
&&\hspace{4.5cm}\cdot
\left[S^x_{\theta}(t_0,x_0+\eta)-S^x_\theta (t_0,x_0)\right].
\end{eqnarray*}
Note that the product rule holds for the RHS derivative as usual, cf. Lemma \ref{product_rule}. Furthermore, we use the comparison lemma in the version for RHS limits; a more general version for Dini derivatives can be found in \cite[Lemma 3.4]{khalil}. This shows that the result of \cite[Proposition 3.1]{zaa} remains true for the Riemannian metric $M$ as in Definition \ref{Riem}.

\vspace{0.3cm}
\noindent {\bf Floquet exponent}

In this part of the proof in \cite{zaa} we need to show that
\begin{eqnarray*}
\lefteqn{\int_0^T \left(\ln \left(p(t)^T M(t,x(t))p(t)\right)\right)'_+\,dt }\\&=&
\ln \left(p(T)^T M(T,x(T))p(T)\right)-\ln \left(p(0)^T M(0,x(0))p(0)\right)
\end{eqnarray*}
where $x(t)$ is the periodic orbit and $p(t)e^{-\nu_0 t}$ with $p(0)=p(T)\not=0$ is a solution of the first variation equation $\dot{y}=D_xf(t,x(t))y$ along the periodic orbit.
Note that we have already used that the RHS derivative satisfies the product rule, cf. Lemma \ref{product_rule}.
 Using the same lemma, the composition with the differentiable function $\ln$ is also RHS differentiable. To apply the last statement of this lemma, we need to show that
 $$\int_0^T \left(\ln \left(p(t)^T M(t,x(t))p(t)\right)\right)'_+\,dt$$
 exists and is finite.
 Indeed, we have, using the chain rule in Lemma \ref{product_rule}, that
$$  \left(\ln \left(p(t)^T M(t,x(t))p(t)\right)\right)'_+=
 \frac{ \left(p(t)^T M(t,x(t))p(t)\right)'_+}{p(t)^T M(t,x(t))p(t)}.$$
Since $M(t,x(t))$ is positive definite and continuous and $p(t)\not=0$, the denominator is bounded away from zero. We apply the product rule to the nominator and use
that $$\int_0^T p(t)^T M'_+(t,x(t))p(t)\, dt$$ exists and is finite by assumption, and that the other terms are smooth. This shows the theorem.
\end{proof}

\section{A solution of the semidefinite optimization problem defines a contraction metric}
\label{sec2}

In this section we will state a semidefinite optimization problem and show that a feasible solution of the semidefinite optimization problem defines a contraction metric, satisfying the assumptions of Theorem \ref{th1}. The solution will be a CPA matrix-valued function $M$, which is defined by the values at the vertices.

We will first define a triangulation of the phase space and then define a Riemannian metric by its values at the vertices of this triangulation and affine on each simplex. This  CPA Riemannian metric is shown to fulfill all conditions of Theorem \ref{th1}, also inside the simplices.

\subsection{Triangulation}\label{tria}
%

For the algorithm to construct a piecewise affine Lyapunov function we need to fix our triangulation.  This triangulation is
a subdivision of $S^1_T \times \R^n$ into $(n+1)$-simplices, such that the intersection of any two different simplices in the subdivision is either empty
or a $k$-simplex, $0\leq k < n+1$, and then its vertices are the common vertices of the two different  simplices.
Such a structure is often referred to as a simplicial $(n+1)$-complex.

In contrast to \cite{HAFSTEIN-revised-cpa}, we do not need to have a fine triangulation near 0, but we need to ensure that the triangulation respects the periodicity in $t$ of the phase space, i.e. it is a triangulation of the cylinder $S_T^1\times \mathbb R^n$.
%

For the construction we use  the
standard orthonormal basis $\be_0,\be_1,\be_2,\ldots,\be_n$ of $\mathbb R\times \R^n$, where $\be_0$ denotes the unit vector in $t$-direction.
We also fix a scaling matrix $S=\diag(1,s_1,\ldots,s_n)$ with diagonal entries $s_i>0$ which fixes the ratio of the fineness of the triangulation with respect to different directions. We have fixed the $1$ in $t$-direction to ensure that the simplicial complex is compatible with the periodicity.

Further, we use  the characteristic functions $\chi_\cJ(i)$
equal to one if $i\in\cJ\subset \mathbb N$ and equal to zero if $i\notin\cJ$ and the functions $\bR^\cJ:\R^{n+1} \to \R^{n+1}$, defined for
every $\cJ\subset\{1,2,\ldots,n\}$ by
$$
\bR^\cJ(\btx) :=x_0\be_0+ \sum_{i=1}^n(-1)^{\chi_\cJ(i)}x_i\be_i,
$$
where $\tx=(x_0,x_1,\ldots,x_n)$ and $x_0=t$.
Thus $\bR^{\cJ}(\btx)$ puts a minus in front of the coordinate $x_i$ of $\btx$ if $i\in\cJ$.

 \begin{definition}
 \label{neig}
 Denote by $\cN$ the set of all subsets $\cD \subset S^1_T\times \R^n$  that fulfill\,{\rm :}
 \begin{enumerate}
   \item[i)] $\cD$ is compact.
   \item[ii)] The interior $\cD^\circ$ of $\cD$ is  connected and open.
   \item[iii)] $\cD=\overline{\cD^\circ}$.
 \end{enumerate}
Note that compactness, connectedness etc. refer to the space $S^1_T\times \R^n$.
 \end{definition}

\begin{figure}[here]
	\centerline{
\includegraphics[height=3.8cm]{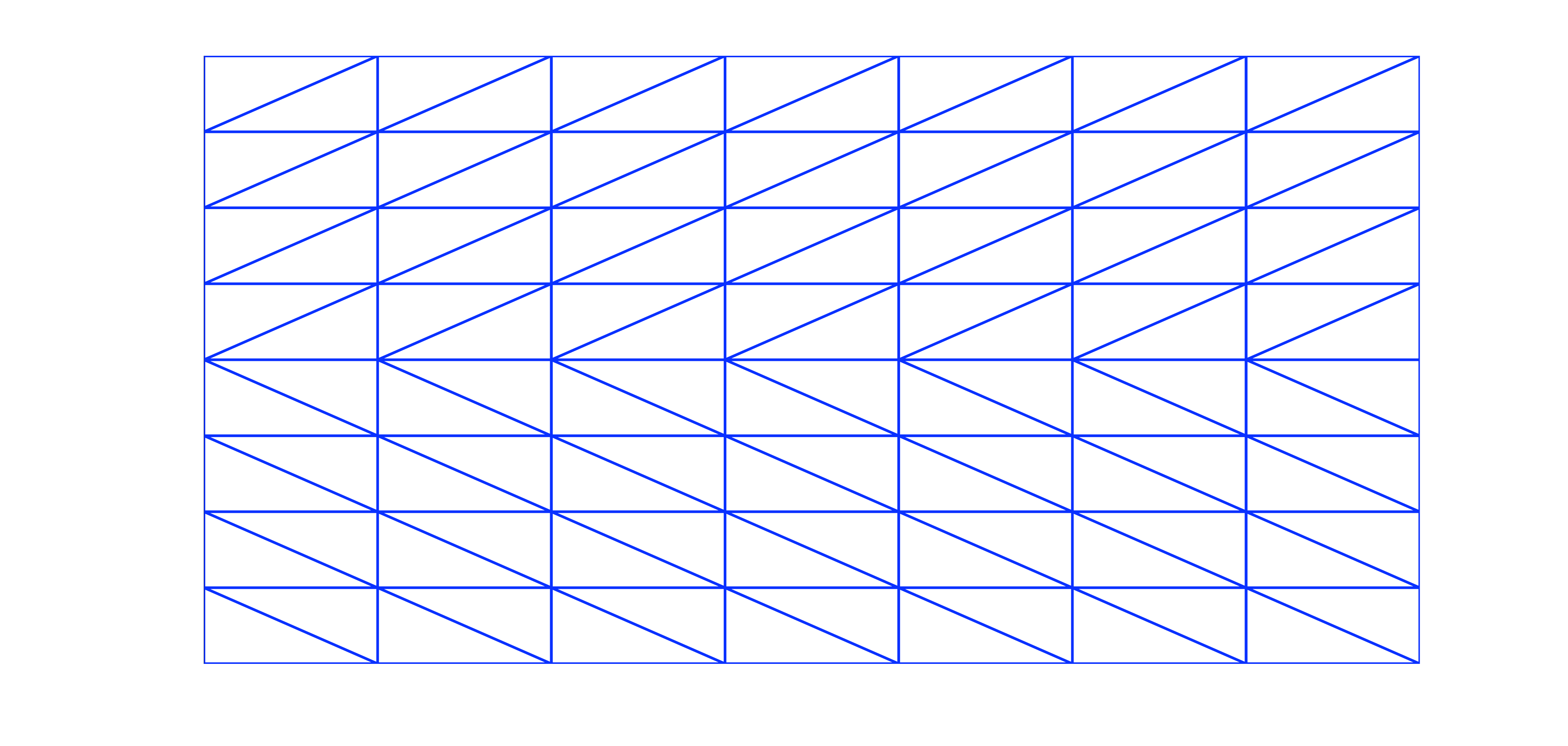}}
\caption{The triangulation $\cT^\text{basic}$ of $S^1_T\times \R^1$. Note that the operator $\bR^\cJ$ puts a minus-sign in $x$-directions, but not in the $t$-direction, which results in the simplices shown in this  figure.}
\label{fig1}\end{figure}

\begin{definition}\label{triconstr}
Let $\cC\in\cN$ be a given subset of $S^1_T\times \R^n$.  We will define a triangulation $\cT^\cC_{K}$ of a $\cD_K\in\cN$, $\cD_K\supset\cC$, that approximates $\cC$.
To construct the triangulation of a set $\cT^\cC_{K}$, we first define the triangulations $\cT^\text{basic}$ and $\cT_K^\text{basic}$ as intermediate steps.
\begin{enumerate}
  \item The triangulation $\cT^\text{basic}$, cf. Figure \ref{fig1}, consists of the simplices
  $$\T_{\btz,\cJ,\sigma}:=
 \co\left(\bR^\cJ\left(\btz+\sum_{i=0}^j\be_{\sigma(i)}\right)\,:\,j=-1,0,1,2,\ldots,n\right)
  $$
  for all $\btz\in\N^{n+1}_{0}$, all $\cJ\subset \{1,2,\ldots,n\}$, and all $\sigma\in S_{n+1}$, where  $S_{n+1}$ denotes the set of  all permutations of the
numbers $0,1,2,\ldots,n$.
  \item
  Now
   scale down the triangulation  $\cT^\text{basic}$ with the mapping
  $\tx\mapsto \rho\,S \tx$, where $\rho:= 2^{-K}T$ and $S$ is the fixed diagonal matrix defined above.
  We denote by $\cT^\text{basic}_{K}$ the resulting set of $(n+1)$-simplices, i.e.
  $$\T:=\co\left(\rho S \bR^\cJ\left(\btz+\sum_{i=0}^j\be_{\sigma(i)}\right)\,:\,j=-1,0,1,2,\ldots,n\right).
   $$
   Note that for each simplex $\T\in \cT^{basic}_K$ there is a unique $i\in \mathbb N_0$ such that $\T\in [iT,(i+1)T]\times \mathbb R^n$.
This follows from the fact that the scaling matrix $S$ has $1$ as its entry in $t$-direction and $\rho=2^{-K}T$. Hence, we can and will interpret
 a simplex $\cT\in \cT^{basic}_K$ as a set in $S^1_T \times \mathbb R^n$ in the following step.
\item
  As a final step define
  $$\cT^\cC_{K}:=\left\{\T\in \cT^\text{basic}_{K}\,:\,\T\cap\cC^\circ \neq \varnothing\right\}$$
  and set
  $$\cD_K:=\bigcup_{\T\in\cT^\cC_{K}}\T
  \subset S_T^1\times \mathbb R^n.$$
  Note that $\cT^\cC_K$ consists of finitely many simplices due to the fact that $\cC$ is compact and the triangulation respects the periodicity.
\end{enumerate}
\end{definition}
\begin{lemma}
Consider the sets $\cC$ and $\cD_K$ from the last definition.  Then $\cD_K\supset\cC$ and $\cD_K\in\cN$.
\end{lemma}
\begin{proof}
$\cD_K$ is a closed set containing $\cC^\circ$ and thus contains $\cC\in\cN$ because $\cC=\overline{\cC^\circ}$ by property iii) in Definition \ref{neig}, so $\cC$ is the smallest closed set containing $\cC^\circ$.
That $\cD_K$ fulfilles  properties i) and iii) of Definition \ref{neig} follows directly from the fact that $\cD_K$ is a finite union of $(n+1)$-simplices.  To see that property ii) of Definiton \ref{neig} is also fulfilled,
i.e.~that $\cD_K^\circ$ is connected,
 notice the following:  The definition of $\cT^\cC_{K}$ implies that for any $\T\in\cT^\cC_{K}$ we have $\T^\circ\cap \cC^\circ \neq \varnothing$.
  Because of this, any $\btx\in \T$ can be connected to a $\bty\in\cC^\circ$ with a line contained in $\T^\circ$ with a possible exception of the endpoint $\btx$.
Because $\cC\in\cN$ we have that $\cC^\circ$ is connected so this implies that $\cD_K^\circ$ is connected too, i.e.~$\cD_K$ also fulfilles property ii) of Definition \ref{neig} and therefore $\cD_K\in\cN$.\end{proof}



\begin{remark}
The triangulation $\cT^\text{basic}$ is studied in more detail in sections 4.1 and 4.2 in \cite{siggidiss}.
A sometimes more intuitive description of $\T_{\btz,\cJ,\sigma}$ is the simplex $\{\btx\in\R^{n+1}\,:\,0 \le \tx_{\sigma(0)} \le \ldots \le \tx_{\sigma(n)} \le 1\}$
translated by $\btx \mapsto \btx+\btz$ and then a minus-sign is put in front of the $i$-th entry of the resulting vectors whenever $i\in\cJ\subset \{1,\ldots,n\}$; but no change of sign in the $t$-coordinate.
\end{remark}

$\cT^\cC_{K}$ is truly a triangulation, i.e.~two different simplices in $\cT^\cC_{K}$ intersect in a common face or not at all, as shown in Lemma \ref{simplemma}.

\begin{lemma}
\label{simplemma}
Consider the set of simplices $\cT^\cC_{K}$ from Definition \ref{triconstr} and let $\T_1=\co(\btx_0,\btx_1,\ldots,\btx_{n+1})$ and
$\T_2=\co(\bty_0,\bty_1,\ldots,\bty_{n+1})$ be two of its simplices.   Then
\begin{equation}
\label{simpeq}
\T_1\cap \T_2 = \T_3 :=\co(\btz_0,\btz_1,\ldots,\btz_m),
\end{equation}
where $\btz_0,\btz_1,\ldots,\btz_m$ are the vertices that are common to $\T_1$ and $\T_2$, i.e.~$\btz_i = \btx_{\alpha(i)} = \bty_{\beta(i)}$
for $\alpha,\beta\in S_{n+2}$ and $i=0,\ldots,m$, $m\in \{-1,0,\ldots,n+1\}$.
\end{lemma}

\begin{proof} The
equation (\ref{simpeq}) follows as in Theorem 4.11 in \cite{siggidiss} and the fact that the triangulation respects the periodicity.
\end{proof}

One important property of the chosen triangulation is that the simplices are sufficiently regular, i.e. the angles all have a lower bound. To make this precise and to also measure the influence of $K$, we prove the following lemma. The matrix $X_{K,\nu}$, as defined in the next lemma, consists of the $n+1$ vectors which span the simplex. We obtain an estimate on its inverse $X^{-1}_{K,\nu}$ depending on $K$. This estimate will later be used to estimate the derivative of an affine function on the simplex.

\begin{lemma}
\label{Xlemma}Using the notation of Definition \ref{triconstr},
there is a constant $X^*$, which is independent of $K$ and $\nu$, such that for all simplices
$\T_\nu\in \cT_K^{basic}$ we have
$$\|X_{K,\nu}^{-1}\|_1\le \frac{2^{K}}{s^*T}X^*$$
where
$X_{K,\nu}=\left(\begin{array}{c}(\tx_1-\tx_0)^T\\(\tx_2-\tx_0)^T\\ \vdots\\ (\tx_{n+1}-\tx_0)^T\end{array}\right)$, $\tx_0,\ldots,\tx_{n+1}$ are the vertices of $\T_\nu$ (in any order) and
$s^*:=\min(1,s_1,\ldots,s_n)>0$.\end{lemma}
\begin{proof}
Every simplex in $\cT_K^{basic}$ is formed from an
 $(n+1)$-simplex $\co(\btx_0,\btx_1,\ldots,\btx_{n+1})\in\cT^\text{basic}$
 with corresponding matrix $X_\nu$.
 The matrices $X_{K,\nu}$ and $X_\nu$ relate via
 $$X_{K,\nu}=\rho X_\nu S.$$

 Note that for the matrices $X_\nu$, up to translations, there are finitely many different simplices in $\cT^\text{basic}$ and
also finitely many possibilities of ordering the vertices of any such simplex. Hence, there is only a finite number of possibilities of forming such a matrix $X_\nu$.  Further, all of them are invertible.
 This means that we can define $\alpha^2>0$ as the minimal eigenvalue of all possible $X_\nu^TX_\nu$. Note that
 $$\lambda_{min}(S^TX_\nu^TX_\nu S)\ge \lambda_{min}(X_\nu^TX_\nu) \lambda_{min}(S^TS).$$
 Hence,
 \begin{eqnarray*}
 \|X^{-1}_{K,\nu}\|_1&\le& \sqrt{n+1}\|X_{K,\nu}^{-1}\|_2\\
& =&\sqrt{\frac{n+1}{\lambda_{min}(X_{K,\nu}^TX_{K,\nu})}}\\
 &\le& \sqrt{\frac{n+1}{\lambda_{min}(X_{\nu}^TX_{\nu})\rho^2 (s^*)^2}}\\
& \le&  \frac{ \sqrt{n+1}}{ s^*\alpha T }2^K
 \end{eqnarray*}
as $\rho=2^{-K}T$.
 This shows the lemma with $X^*= \frac{ \sqrt{n+1}}{ \alpha }$. Note especially that $X^*$ is a constant independent of $K$ and $\nu$.
\end{proof}

In the next lemma we show that a CPA function on a simplicial complex as defined above satisfies the assumptions of Lemma \ref{ableitung} and the technical assumption of Theorem \ref{th1}.

\begin{lemma}\label{le4.9}
Let $\cT$ be a simplicial complex in $S_T^1\times \mathbb R^n$, which is locally finite, i.e.~each point has a neighborhood $U$ such that $U\cap \T\not=\varnothing $ only for a finite number of simplices $\T\in \cT$. Denote $\cD=\cup_{\T\in\cT}\T$.
Let $M\in C^0(\cD,\mathbb S^n)$ be a CPA function, which is affine on each simplex of $\cT$, and let $M(t,x)$ be positive definite for all $(t,x)$.

Then $M$ is Lipschitz-continuous on $\cD^\circ$ and  for each $(t_0,x_0)\in \cD^\circ$
\begin{eqnarray}\lim_{\theta\to 0^+}\frac{M( (t_0,x_0)+\theta \tf(t_0,x_0))-M(t_0,x_0)}{\theta}\label{eq*}
\end{eqnarray}
 exists.

 There also exists (at least) one simplex $\T_\nu\in \cT$ and $\theta^*>0$ such that
\begin{eqnarray}
(t_0,x_0)+\theta \tf(t_0,x_0)\in \T_\nu\mbox{ for all }\theta\in [0,\theta^*].
 \label{T}
 \end{eqnarray}
 $M|_{\T_\nu}$ restricted to this simplex is an affine function and
 the expression in (\ref{eq*}) is equal to
 \begin{itemize}
 \item
$M|_{\T_\nu}'(t_0,x_0)=\nabla_{\tx} M|_{\T_\nu}(t_0,x_0)\cdot \tf(t_0,x_0)$,  the smooth orbital derivative of the affine function $M|_{\T_\nu}$,
\item and $M'_+(t_0,x_0)$.
\end{itemize}
These expressions are the same for all simplices which satisfy (\ref{T}).

In particular, $M$ is  a Riemannian metric in $\cD^\circ$ in the sense of Definition \ref{Riem}.
Moreover,
 $$\int_0^T p(t)^T M'_+(t,x(t))p(t)\, dt$$ exists and is finite for all solutions
$x(t)$ with $x(t)\in \cD^\circ$ for all $t\in [0,T]$  and all functions $p\in C^0([0,T], \mathbb R^n)$.
\end{lemma}
\begin{proof}
Let $\tx_0:=(t_0,x_0)\in\cD^\circ$.
Since there are only finitely many simplices $\T_1,\ldots,\T_N\in \cT$, which have a non-empty intersection with a neighborhood $U$ of $\tx_0$, and on each of them $M_{ij}$ is affine and has a finite Lipschitz constant, the overall constant can be chosen as the maximum of the finitely many.

Now we show that there is a $\theta^*>0$ and a simplex $\T_\nu\in \cT$ such that
$\tx_0+\theta \tf(\tx_0)\in \T_\nu$ for all $\theta\in [0,\theta^*]$. Then it is clear by the smooth chain rule that the limit (\ref{eq*}) exists and is equal to the smooth orbital derivative $M|_{\T_\nu}'(t,x)=\nabla_{\tx} M|_{\T_\nu}(\tx_0)\cdot \tf(\tx_0)$ of the function $M$ restricted to the simplex $\T_\nu$. Furthermore,  (\ref{eq*}) is equal to $M'_+(\tx_0)$ by Lemma \ref{ableitung}. If there are two simplices $\T_1$ and $\T_2$ with property (\ref{T}), then $M|_{\T_1}(\tx_0)=M|_{\T_2}(\tx_0)$ and also $(M|_{\T_1})'_+(\tx_0)=(M|_{\T_2})'_+(\tx_0)$ by Lemma \ref{ableitung} since $M|_{\T_1}(\tx)=M|_{\T_2}(\tx)$ for all $\tx=\tx_0+\theta\tf(\tx_0)$ with $\theta\in [0,\theta^*]$.

Now we show that there exists a simplex with property (\ref{T}).
Indeed, there is a $J\in \mathbb N$ such that
$\tx_0+\frac{\tf(\tx_0)}{j}\in U$ for all $j\ge J$ where $U$ is the neighborhood of $\tx_0$ from above.
Assume that there is no simplex $\T_\nu$ and no $\theta^*>0$ such that $\tx_0+\theta \tf(\tx_0)\in \T_\nu$ for all $\theta\in [0,\theta^*]$. Then
$$\left\{\tx_0+\frac{\tf(\tx_0)}{j}, j\ge J  \right\}\cap \T_k=\varnothing$$
for all simplices $\T_1,\ldots,\T_{N^*}\in \cT$ to which $\tx_0$ belongs, since if one such point was in $\T_k$, then, due to the convexity of $\T_k$, the whole line between that point and $\tx_0$ would be in $\T_\nu$.

Since there are infinitely many points in $\left\{\tx_0+\frac{\tf(\tx_0)}{j}, j\ge J  \right\},$ which are in $U$, but only finitely many simplices that have nonempty intersection with $U$ by assumption, at least one of them, say $\T_\nu$, must contain infinitely many such points. Since the sequence $\tx_0+\frac{\tf(\tx_0)}{j}$ converges to $\tx_0$ as $j\to\infty$ and $\T_\nu$ is closed, $\tx_0\in \T_\nu$ and hence $\T_\nu$ must be one of the $\T_1,\ldots,\T_{N^*}$ defined above, which is a contradiction. This shows the statement.

For the last statement, we show that the function $\theta\to M'_+(S_\theta \tx_0)$ is RHS continuous.
Then the function $p(t)^T M'_+(t,x(t))p(t)$ is RHS continuous and bounded, and thus it is integrable. 

Fix $\tx_0$ and a neighborhood $U$ such that there are only finitely many simplices with  nonempty intersection with $U$. Denote by $\T_1,\ldots,\T_N$ the subset of those finitely many simplices such that for each $\T_i$ there is a $\theta_i^*$ with $\tx_0+\theta \tf(\tx_0)\in \T_i$ for all $\theta \in [0,\theta_i^*]$, these are the ones the qualify to be $\T(\tx_0)$.

Now take a sequence $\theta_k\to 0^+$ and we seek to prove that $M'_+(S_{\theta_k} \tx_0)\to M'_+(\tx_0)$. Since $S_{\theta_k}\tx_0\in U$, if $k$ is large enough, there are only finitely many simplices which contain infinitely many elements $S_{\theta_k}\tx_0$ of the sequence. We show that these simplices are in fact a subset of $\T_1,\ldots,\T_N$ as defined above, satisfying property (\ref{T}). If this was not true then there would be a simplex $\T$ and a sequence (the subsequence from above) of points $S_{\theta_k}\tx_0\in \T$,  $\tx_0\in \T$, but $\tx_0+\theta \tf(\tx_0)\not\in\T$ for all $\theta>0$. The simplex $\T$ is the intersection of finitely many half-spaces. The point $\tx_0$ lies on the boundary of (at least) one of these half-spaces, since otherwise a neighborhood of $\tx_0$ would lie in $\T$.
Hence, there is a half-space of the form  $\{\tx\mid (\tx-\tx_0)\cdot \tn\le 0\}$ with $\tn\in\mathbb R^{n+1}\setminus \{0\}$ such that $\tf(\tx_0)\cdot \tn>0$. Hence there is an $\epsilon$ such that $\tv\cdot \tn>0$ for all $\tv$ with $\|\tv-\tf(\tx_0)\|<\epsilon$. Now we have $$\T\ni S_{\theta_k}\tx_0=\tx_0+\theta_k \int_0^1\tf(\tx_0+\lambda(S_{\theta_k}\tx_0-\tx_0))\,d\lambda.$$
If $k$ is large enough, then
$\tv:=\int_0^1\tf(\tx_0+\lambda(S_{\theta_k}\tx_0-\tx_0))\,d\lambda$ fulfills $\|\tv-\tf(\tx_0)\|<\epsilon$ by continuity of $\tf$ and thus $S_{\theta_k}\tx_0\not\in\T$, which is a contradiction. This shows that each simplex which contains infinitely many elements of the sequence fulfills property (\ref{T}) and thus is a candidate for $\T(\tx_0)$.

Let $\T$ be one of the simplices  which contains infinitely many elements of the sequence and define the subsequence $\theta_{k_l}$ by choosing the next $\theta_k$ with the property  $S_{\theta_k}\tx_0\in \T$. We do the same for all (finitely many) simplices that contain infinitely many elements.
For the convergence of the overall sequence it is enough to show that every subsequence $S_{\theta_k}\tx_0$ converges to the same limit, since there are only finitely many.

For each subsequence, the elements are in one simplex and on this simplex $M'_+$ is a continuous function, so it converges. The limit is the same, as $M'_+(\tx_0)$ is the same no matter which simplex $\T(\tx_0)$ we choose.
\end{proof}

\begin{remark}
The condition locally finite  for the triangulation is indispensable as shown by the following  example:  Let $f(1/2,0)=(0,1)$ and for every $n\in\N$ define the triangle
$\T_n:=\co((0,0),(1,1/n),(1,1/(n+1)))$.  Then clearly there is no $n\in \N$ with a corresponding $\theta^*>0$ such that $(1/2,0) +\theta f(1/2,0)=(1/2,\theta)\in \T_n$ for all
$\theta\in[0,\theta^*]$.
\end{remark}

\subsection{The semidefinite optimization problem}

 For each simplex $\T_\nu\in\cT^\cC_K$ we denote
$$
h_\nu\ :=\diam (\T_\nu)=\max_{\tx,\ty\in \T_\nu}\|\tx-\ty\|_2.$$
Note that for our triangulation we have with $\T_\nu=\co (\tx_0,\ldots,\tx_{n+1})$ and $S^*=\sqrt{n+1}\max(1,s_1,\ldots,s_n)$ the estimate
\begin{eqnarray}
h_\nu=\max_{k,l\in\{0,\ldots,n+1\}} \|\btx_k-\btx_l\|_2
\le S^*\rho=S^*2^{-K}T.\label{hnu}
\end{eqnarray}
 Moreover, denote
\begin{eqnarray}
B_{\nu}&:=&\max_{\tx\in \T_\nu, i,j\in \{0,\ldots,n\}}\left\|\frac{\partial^2 f(\tx)}{\partial x_i\partial x_j}\right\|_\infty,
\mbox{ where }x_0:=t\\
B_{3,\nu}&:=&\max_{\tx\in \T_\nu, i,j,k\in \{0,\ldots,n\}}\left\|\frac{\partial^3 f(\tx)}{\partial x_i\partial x_j\partial x_k}\right\|_\infty\mbox{ if }f\in C^3.
\end{eqnarray}

\vspace{0.3cm}
\noindent
{\bf Variables}

The variables of the semidefinite optimization problem are
\begin{enumerate}
\item $M_{ij}(\tx_k)\in \mathbb R$ for all $1\le i\le j\le n$ and all vertices $\tx_k$ of all simplices
$\T_\nu=\co(\tx_0,\ldots,\tx_{n+1})\in \cT^\cC_K$ -- values of the Riemannian metric at vertices
\item $C_\nu \in \mathbb R_0^+$ for all simplices $\T_\nu\in \cT^\cC_K$ -- bound on $M$ in $\T_\nu$
\item $D_\nu \in \mathbb R_0^+$ for all simplices $\T_\nu\in \cT^\cC_K$ -- bound on  derivative of $M$ in $\T_\nu$

\end{enumerate}
Thus we have $2s+\frac{1}{2}n(n+1)v$ variables, where $s$ denotes the number of simplices and $v$ the number of vertices.

In this section, there is no objective function, so we are considering a feasibility problem -- see Section \ref{objective} for a suitable objective function. The equality Constraint 1. can be incorporated by choosing the variables in this way and thus is no actual constraint.
 Note that Constraint 3. is linear and Constraints 2., 4. and 5. are semidefinite. Constraints 2. and  5. need to be satisfied for each simplex and then for each vertex, so vertices common to more simplices need to satisfy several constraints. In Constraint 4., however, each vertex only needs to be checked once. Note, however, that we can replace the individual constants $C_\nu$ for each simplex by their maximum $C$, so that also Constraint 2. only needs to be checked at each vertex. Note that if the triangulation is fine enough and the system has an exponentially stable periodic orbit, this more restrictive form of the constraint can be satisfied, cf. the proof of Theorem \ref{feas}.
%

\begin{enumerate}
\item {\bf Periodicity}

$$ M_{ij}(0,x_k)=M_{ij}(T,x_k)$$ for all $1\le i\le j\le n$ and for all vertices at times $0$ and $T$, i.e. for all vertices $(0,x_k)$ and $(T,x_k)$. Note that by construction of the triangulation $(0,x_k)$ is a vertex of a simplex in $\cT^\cC_K$ if and only if $(T,x_k)$ is.

\item {\bf Bound on $M$}

$$M(\tx_k)\preceq C_\nu I$$ for all vertices
$\tx_k$ of all simplices $\T_\nu=\co(\tx_0,\ldots,\tx_{n+1})\in\cT^\cC_K$, where the symmetric matrix $M(\tx_k)$ is defined by setting $M_{ji}(\tx_k):=M_{ij}(\tx_k)$ for all $1\le i< j\le n$.

\item  {\bf Bound on derivative of $M$}

$$|(w^{\nu}_{ij})_l |\le \frac{D_{\nu}}{n+1}$$ for all $l=0,\ldots,n$,
$1\le i\le j\le n$ and for all simplices $\T_\nu\in\cT^\cC_K$, where $w^{\nu}_{ij}=\nabla_{\tx} M_{ij}\big|_{\T_\nu}(\tx)$ for all $\tx\in \T_\nu=\co(\tx_0,\ldots,\tx_{n+1})$, which is given by
\begin{eqnarray}w^{\nu}_{ij}&:=&X^{-1}_{K,\nu}\left(\begin{array}{c}M_{ij}(\tx_1)-M_{ij}(\tx_0)\\ \vdots\\ M_{ij}(\tx_{n+1})-M_{ij}(\tx_0)\end{array}\right)\in\mathbb R^{n+1}\label{defw}\end{eqnarray}
 where
$X_{K,\nu}=\left(\begin{array}{c}(\tx_1-\tx_0)^T\\(\tx_2-\tx_0)^T\\ \vdots\\ (\tx_{n+1}-\tx_0)^T\end{array}\right)\in\mathbb R^{(n+1)\times (n+1)}$.

\item {\bf Positive definiteness of $M$}

We fix $\epsilon_0>0$.
$$M(\tx_k)\succeq  \epsilon_0 I$$ for all vertices $\tx_k\in \T_\nu$ of all simplices $\T_\nu\in\cT^\cC_K$.

\item {\bf Contraction of the metric}

$$M(\tx_k)D_xf(\tx_k)+D_xf(\tx_k)^TM(\tx_k)+(w_{ij}^\nu \cdot \tf(\tx_k))_{i,j=1,\ldots,n}+(E_{\nu}+1)I\preceq  0$$ for all simplices $\T_\nu=\co(\tx_0,\ldots,\tx_{n+1})\in \cT^\cC_K$ and all of its vertices $k=0,\ldots,n+1$.
Here, $\tf(\tx)=\left(\begin{array}{c}1\\ f(\tx)\end{array}\right)$ and $(w_{ij}^\nu \cdot \tf(\tx_k))_{i,j=1,\ldots,n}$ denotes the symmetric $(n\times n)$ matrix with entries $w_{ij}^\nu \cdot \tf(\tx_k)$, where $w^{\nu}_{ij}$ was defined in (\ref{defw}) and is the same vector for all vertices in one simplex
and
\begin{eqnarray}
E_{\nu}&=&\left\{\begin{array}{ll}
h_\nu n B_\nu [\sqrt{n+1}h_\nu  D_\nu+2n(n+1)C_\nu]
&\mbox{ if } f\in C^2\\
h_\nu^2n
[\sqrt{n+1}(1+4n)B_\nu D_\nu +2n(n+1)B_{3,\nu}C_\nu]&\mbox{ if } f\in C^3
\end{array}\right.
\end{eqnarray}
 \end{enumerate}

\begin{remark}
In Constraint 3. we have claimed that the gradient of the affine function $M_{ij}\big|_{\T_\nu}$, i.e. $\nabla_{\tx} M_{ij}\big|_{\T_\nu}=w_{ij}^\nu$, is given by the expression in (\ref{defw}). For a proof of this fact
 and, moreover, that the definition  is independent of the choice of the vertex $\tx_0$, see \cite[Remark 2.9]{HAFSTEIN-revised-cpa}.
\end{remark}

\begin{remark}
The constraints above are easily transferred into the standard form $\sum_{i=1}^m F_i y_i-F_0\succeq 0$ in the following way: Denote by $y_1,\ldots,y_{n(n+1)/2}$ the matrix elements $M_{ij}(\tx_1)$, $1\le i\le j\le n$, by the following $\frac{n(n+1)}{2}$ elements $y_i$ the matrix elements of $M_{ij}(\tx_2)$ etc. and finish the vector $y$ by the $C_\nu$ and $D_\nu$. This results in $m=2s+\frac{n(n+1)}{2}v$ as above.

Now Constraints 2. and 4. are expressed in $(n+2)s$, $v$ blocks of size $n$ each, respectively, in the matrices $F_i$. Constraint 3. is expressed in $2\cdot \frac{n(n+1)}{2}$ conditions (for each $i$, $j$) for each of the $n+1$ entries of the vector $w_{ij}^\nu$. This needs to be considered for each simplex, so that we have $n(n+1)^2s$ blocks of size $1$.

Finally, Constraint 5. is expressed in $(n+2)s$ blocks of size $n$, since for every simplex each of its vertices needs to be considered. Note that $E_\nu$ depends linearly on $C_\nu$ and $D_\nu$  and $w_{ij}^\nu$ depends linearly on $M_{ij}(\tx_k)$.

The size of the matrices $F_i$ is thus a block-diagonal structure with $n(n+1)^2s$ blocks of size $1$ and $v+2(n+2)s$ blocks of size $n$.
\end{remark}

\begin{remark}\label{con_new}
Note that Constraint 2. implies  that $\max_{k=0,\ldots,n+1} |M_{il}(\tx_k)|\le \|M(\tx_k)\|_{max}\le\|M(\tx_k)\|_2\le  C_\nu$ since $M(\tx_k)$ is positive definite.
Note also that Constraint 3. implies that $\|w_{ij}^\nu\|_1\le D_\nu$. Moreover, Constraint 5. is equivalent to
$$\lambda_{max}\left(M(\tx_k)D_xf(\tx_k)+D_xf(\tx_k)^TM(\tx_k)+(w_{ij}^\nu \cdot \tf(\tx_k))_{i,j=1,\ldots,n}\right)+E_{\nu}\le  -1$$
where $\lambda_{max}$ denotes the maximal eigenvalue.
\end{remark}

\subsection{Feasible solution is CPA contraction metric}

A solution of the semidefinite optimization problem returns a matrix $M_{ij}(\tx_k)$ at each vertex $\tx_k$ of the triangulation for $1\le i\le j\le n$. We define the CPA metric by affine interpolation on each simplex.

\begin{definition}\label{CPA-def}
Fix $\cC$ and a triangulation $\cT^\cC_K$ with $\cD_K=\bigcup_{\T_\nu\in \cT^\cC_K}\T_\nu$.
Let $M_{ij}(\tx_k)$  be defined by a feasible solution of the semidefinite optimization problem. Let $(t,x)=\tx\in \T_\nu=\co(\tx_0,\ldots,\tx_{n+1})$ such that  $\tx=\sum_{k=0}^{n+1}\lambda_k\tx_k$ with $\lambda_k\in [0,1]$ and $\sum_{k=0}^{n+1}\lambda_k=1$.
Then define
$$M(\tx)=\sum_{k=0}^{n+1}\lambda_kM(\tx_k).$$
\end{definition}

\begin{lemma}\label{le1}
The matrix $M(\tx)$ as in Definition \ref{CPA-def} is symmetric and positive definite for all $\tx\in\cD_K$. The function $M(t,x)$ is periodic in $t$ with period $T$.
\end{lemma}
\begin{proof}
The symmetry follows directly from the symmetry of $M(\tx_k)$:
$$M_{ij}(\tx)=\sum_{k=0}^{n+1}\lambda_kM_{ij}(\tx_k)=\sum_{k=0}^{n+1}\lambda_kM_{ji}(\tx_k)
=M_{ji}(\tx).$$

The positive definiteness also follows from the positive definiteness of
$M(\tx_k)$, using that the minimal eigenvalue $\lambda_{\min}$ is a concave function. Indeed, consider $\tx\in \T_\nu=\co(\tx_0,\ldots,\tx_{n+1})$.  Note that $\lambda_{min}(M(\tx_k))\ge \epsilon_0$  for all $k=0,\ldots,n+1$ due to Constraint 4. Thus,
\begin{eqnarray*}
\lambda_{min}(M(\tx))&=&\lambda_{min}\left(\sum_{k=0}^{n+1}\lambda_kM(\tx_k)\right)\\
&\ge &\sum_{k=0}^{n+1}\lambda_k\lambda_{min}(M(\tx_k))\\
&\ge&\epsilon_0\sum_{k=0}^{n+1}\lambda_k\\
&=&\epsilon_0>0.
\end{eqnarray*}
The $T$-periodicity follows directly from the definition and the triangulation:
$$M(0,x)=\sum_{k=0}^{n+1}\lambda_kM(0,x_k)=\sum_{k=0}^{n+1}\lambda_kM(T,x_k)
=M(T,x).$$
\end{proof}

We will now relate $M'(\tx)$ to $M'(\tx_k)$, as well as $M(\tx)D_xf(\tx)$ to $M(\tx_k)D_xf(\tx_k)$. For the proof we will need the following auxiliary result, see \cite[Proposition 4.1 and Corollary 4.3]{robertlarssiggi2}.

\begin{lemma}\label{help}
Let $f\in C^2(\mathbb R^{n+1},\mathbb R^n)$ and $\tx\in \co(\tx_0,\ldots,\tx_{n+1})=\T_\nu$ with  $\tx=\sum_{k=0}^{n+1}\lambda_k\tx_k$, $\lambda_k\in [0,1]$ and $\sum_{k=0}^{n+1}\lambda_k=1$.

Then  for all $l\in\{1,\ldots,n\}$ we have
 $$\left|f_l(\tx)-\sum_{k=0}^{n+1}\lambda_k f_l(\tx_k)\right|\le
 \max_{\tx\in \T_\nu}\|H_{f_l}(\tx)\|_2 h_\nu^2,$$
 especially
$$\left\|f(\tx)-\sum_{k=0}^{n+1}\lambda_k f(\tx_k)\right\|_\infty\le
 (n+1) B_\nu h_\nu^2,$$
where $h_\nu=\diam(\T_\nu)$, $B_\nu=\max_{\tx\in \T_\nu, i,j\in \{0,\ldots,n\}}\left\|\frac{\partial^2 f(\tx)}{\partial x_i\partial x_j}\right\|_\infty$,
$x_0:=t$ and  $H_{f_l}(\tx):=\left(\frac{\partial^2 f_l(\tx)}{\partial x_i\partial x_j}\right)_{i,j=0,\ldots,n}$ denotes the Hessian of $f_l$.
\end{lemma}
%


In the following we restrict ourselves to one simplex $\T_\nu
\in \cT_K^\cC$.
First we need to define $M'(\tx)$ (the orbital derivative) for a general point in the simplex $\T_\nu$. Note that for all points of a simplex the vector $\nabla_{\tx} M_{ij}(\tx)$ is the same, but the contribution of $\tf(\tx)$ is different.

\begin{definition}\label{basic}Let $M(\tx)$ be as in Definition \ref{CPA-def}.
Fix a point $\tx\in\cD^\circ_K$ and a simplex $\T_\nu=\co(\tx_0,\ldots,\tx_{n+1})\in \cT^\cC_K$ such that $\tx+\theta\tf(\tx)\in \T_\nu$ for all $\theta\in [0,\theta^*]$, cf. Lemma \ref{le4.9}. Then $\tx=\sum_{k=0}^{n+1}\lambda_k\tx_k$ with $\lambda_k\in [0,1]$, $\sum_{k=0}^{n+1}\lambda_k=1$ and
$$(M_{ij})'_+(\tx)=M'_{ij}\big|_{\T_\nu}(\tx)=\nabla_{\tx} M_{ij}\big|_{\T_\nu}(\tx)\cdot \tf(\tx)=w_{ij}^\nu \cdot \tf(\tx).$$
 Note that $M_{ij}\big|_{\T_\nu}$ is an affine function and its gradient $w^{\nu}_{ij}$ was defined in Constraint 3. and is the same vector for all points $\tx\in \T_\nu$. Note that for vertices $\tx=\tx_k$ this term appears in Constraint 5. 
\end{definition}

\begin{lemma}\label{4.12}
Let $f\in C^2$ and
let $M(\tx)$ be as in Definition \ref{CPA-def}.
Fix a point $\tx\in\cD^\circ_K$ and a corresponding simplex $\T_\nu\in \cT_K^\cC$ as in Definition \ref{basic}.

Then we have the following estimates
for all $\tx\in\T_\nu$
\begin{eqnarray*}
\left|(M_{ij})_+'(\tx)-\sum_{k=0}^{n+1}\lambda_k w^\nu_{ij}\cdot \tf(\tx_k)\right|
&\le&\sqrt{n+1}
B_\nu D_\nu h_\nu^2\\
\left|
(M(\tx)D_xf(\tx))_{ij}-\sum_{k=0}^{n+1}\lambda_k (M(\tx_k)D_xf(\tx_k))_{ij}\right|&\le&
n(n+1)B_\nu
 C_\nu h_\nu\\
\left|
(D_xf(\tx)^TM(\tx))_{ij}-\sum_{k=0}^{n+1}\lambda_k (D_xf(\tx_k)^TM(\tx_k))_{ij}\right|&\le&n(n+1)B_\nu
C_\nu h_\nu,
\end{eqnarray*}
where $B_\nu=\max_{\tx\in \T_\nu, i,j\in \{0,\ldots,n\}}\left\|\frac{\partial^2 f(\tx)}{\partial x_i\partial x_j}\right\|_\infty$, $x_0:=t$ and  $h_\nu=\diam(\T_\nu)$.

Altogether we have 
$$\begin{array}{l}
\bigg\|
\left(M(\tx)D_xf(\tx)+D_xf(\tx)^TM(\tx)+M'_+(\tx)\right)
\\
\hspace{0.2cm}
-
\sum_{k=0}^{n+1}\lambda_k\left(M(\tx_k)D_xf(\tx_k)+D_xf(\tx_k)^TM(\tx_k)+(w^\nu_{ij}\cdot \tf(\tx_k))_{i,j=1,\ldots,n}\right)\bigg\|_{max}
\end{array}$$

\vspace{-.5cm}
\begin{eqnarray*}
&\le& h_\nu B_\nu (
\sqrt{n+1} h_\nu D_\nu+2n(n+1)  C_\nu)=\frac{E_\nu}{n}.
\end{eqnarray*}

If $f\in C^3$ we obtain in addition the estimates

$$
\left|
(M(\tx)D_xf(\tx))_{ij}-\sum_{k=0}^{n+1}\lambda_k (M(\tx_k)D_xf(\tx_k))_{ij}\right|
$$
\vspace{-.5cm}

$$\hspace{1.5cm}\le
n(2\sqrt{n+1}B_\nu D_\nu +(n+1)B_{3,\nu}C_\nu)h_\nu^2,
$$

$$
\left|
(D_xf(\tx)^TM(\tx))_{ij}-\sum_{k=0}^{n+1}\lambda_k (D_xf(\tx_k)^TM(\tx_k))_{ij}\right|
$$
\vspace{-.5cm}

$$\hspace{1.5cm}\le
n(2\sqrt{n+1}B_\nu D_\nu +(n+1)B_{3,\nu}C_\nu)h_\nu^2,
$$
where $B_{3,\nu}=\max_{\tx\in \T_\nu, i,j,k\in \{0,\ldots,n\}}\left\|\frac{\partial^3 f(\tx)}{\partial x_i\partial x_j\partial x_k}\right\|_\infty$, $x_0:=t$.
Altogether we have 
$$\begin{array}{l}
\bigg\|
\left(M(\tx)D_xf(\tx)+D_xf(\tx)^TM(\tx)+M'_+(\tx)\right)
\\
\hspace{0.2cm}
-
\sum_{k=0}^{n+1}\lambda_k\left(M(\tx_k)D_xf(\tx_k)+D_xf(\tx_k)^TM(\tx_k)+(w^\nu_{ij}\cdot \tf(\tx_k))_{i,j=1,\ldots,n}\right)\bigg\|_{max}
\end{array}$$

\vspace{-.5cm}
\begin{eqnarray*}
&\le& h_\nu^2 [\sqrt{n+1}(1+4n)B_\nu D_\nu+2n(n+1)B_{3,\nu}C_\nu]=\frac{E_\nu}{n}.
\end{eqnarray*}
\end{lemma}
\begin{proof}
\noindent
{\bf Step 1: $M'$}

Fix one simplex $\T_\nu\in \cT^\cC_K$ of the triangulation and let $\tx\in \T_\nu$. By definition of $M_{ij}(\tx)$ as a CPA function, interpolating $M_{ij}(\tx_k)$ at the vertices, we have $\nabla_{\tx} M_{ij}(\tx)=w^{\nu}_{ij}$ for all $\tx\in \T_\nu$, where $M$ is restricted to the simplex, cf. Definition \ref{basic}. Thus, letting $\tx=\sum_{k=0}^{n+1}\lambda_k\tx_k$ with $\lambda_k\in [0,1]$ and $\sum_{k=0}^{n+1}\lambda_k=1$
\begin{eqnarray*}
(M_{ij})'_+(\tx)&=&\nabla_{\tx} M_{ij}\big|_{\T_\nu}(\tx)\cdot \tf(\tx)\\
&=&w^{\nu}_{ij}\cdot \tf(\tx)\\
&=&w^{\nu}_{ij}\cdot \left(\sum_{k=0}^{n+1}\lambda_k \tf(\tx_k)\right)
+w^{\nu}_{ij}\cdot \left(\tf(\tx)-\sum_{k=0}^{n+1}\lambda_k \tf(\tx_k)\right).
\end{eqnarray*}
Hence
$$
\left|(M_{ij})'_+(\tx)
-\sum_{k=0}^{n+1}\lambda_k w^{\nu}_{ij}\cdot  \tf(\tx_k)\right|\le
\|w^{\nu}_{ij}\|_1\left\| f(\tx)-\sum_{k=0}^{n+1}\lambda_k f(\tx_k)\right\|_\infty
$$
since $\left\| \tf(\tx)-\sum_{k=0}^{n+1}\lambda_k \tf(\tx_k))\right\|_\infty=\left\| f(\tx)-\sum_{k=0}^{n+1}\lambda_k f(\tx_k))\right\|_\infty$ as the $t$-component is $0$.

Now we use Lemma \ref{help} for the $C^2$ function $f$, establishing that
$\| f(\tx)-\sum_{k=0}^{n+1}\lambda_k f(\tx_k)\|_\infty\le (n+1) B_\nu h_\nu^2$.
Thus,\begin{eqnarray*}
\left|(M_{ij})'_+(\tx)
-\sum_{k=0}^{n+1}\lambda_k w^\nu_{ij}\cdot \tf(\tx_k)\right|
&\le&\|w^{\nu}_{ij}\|_1 (n+1)B_{\nu} h_{\nu}^2.
\end{eqnarray*}
Using that  $\|w^{\nu}_{ij}\|_1\le D_\nu$ holds from Remark \ref{con_new}, we obtain
\begin{eqnarray*}
\left|(M_{ij})'_+(\tx)
-\sum_{k=0}^{n+1}\lambda_k  w^\nu_{ij}\cdot \tf(\tx_k)\right|
&\le&( n+1)D_\nu  B_\nu h_\nu^2.
\end{eqnarray*}

\noindent
{\bf Step 2: $MD_xf$}

We consider $(M(\tx)D_xf(\tx))_{ij}=\sum_{l=1}^n M_{il}(\tx)(D_xf(\tx))_{lj}$.
We first consider two scalar-valued functions $g$ and $h$, where $g(\tx)=\sum_{k=0}^{n+1}\lambda g(\tx_k)$ and $h$ is $C^1$ in $\tx$. We have, using Taylor expansion for $h$ at $\tx_k$, i.e. $h(\tx)=h(\tx_k)+\nabla h(\tx^*)(\tx-\tx_k)$, where
$\tx^*$ lies on the straight line between $\tx_k$ and $\tx$,
\begin{eqnarray*}
g(\tx)h(\tx)&=&\sum_{k=0}^{n+1}\lambda_k g(\tx_k)h(\tx)\\
&=&\sum_{k=0}^{n+1}\lambda_k g(\tx_k)[h(\tx_k)+\nabla_{\tx}h(\tx^*)(\tx-\tx_k)]
\end{eqnarray*}
\begin{eqnarray*}
\left|g(\tx)h(\tx)-
\sum_{k=0}^{n+1}\lambda_k g(\tx_k)h(\tx_k)\right|
&\le& \max_{\tx^* \in \T_\nu} \|\nabla_{\tx} h(\tx^*)\|_1h_\nu\max_{k=0,\ldots,n+1} |g(\tx_k)|
\end{eqnarray*}
where  $\|\tx-\tx_k\|_\infty\le\|\tx-\tx_k\|_2\le h_\nu$.

Applying this to $g(\tx)= M_{il}(\tx)$ and $h(\tx)=(D_xf(\tx))_{lj}$ we obtain
with $\max_{k=0,\ldots,n+1} |M_{il}(\tx_k)|\le C_\nu$ from Remark \ref{con_new} and
 $\|\nabla_{\tx} (D_xf(\tx^*))_{lj}\|_1\le (n+1) \max_{\tx\in \T_\nu, i,k\in \{0,\ldots,n\},k\not=0}\left\|\frac{\partial^2 f(\tx)}{\partial x_i\partial x_k}\right\|_\infty\le (n+1)B_\nu$ for all $l,j\in\{1,\ldots,n\}$, using that $\tx,\tx_k\in \T_\nu$ and $\T_\nu$ is a convex set. Thus,
\begin{eqnarray*}
\left|
\sum_{l=1}^nM_{il}(\tx)(D_xf(\tx))_{lj}
-\sum_{l=1}^n\sum_{k=0}^{n+1}\lambda_k M_{il}(\tx_k)(D_xf(\tx_k))_{lj}\right|
&\le&n(n+1)B_\nu
h_\nu C_\nu
\end{eqnarray*}
Hence,
\begin{eqnarray*}
\left|(M(\tx)D_xf(\tx))_{ij}-\sum_{k=0}^{n+1}\lambda_k (M(\tx_k)D_xf(\tx_k))_{ij}\right|
&\le&n(n+1)B_\nu
h_\nu C_\nu.
\end{eqnarray*}
A similar estimate holds for  $D_xf(\tx)^TM(\tx)$.

\noindent
{\bf Step 2': $MD_xf$}

If $f\in C^3$, then we can derive an estimate for this term which establishes order $h_\nu^2$.   We consider two scalar-valued functions $g$ and $h$, where $g(\tx)=\sum_{k=0}^{n+1}\lambda_k g(\tx_k)$ for $\tx=\sum_{k=0}^{n+1}\lambda_k\tx_k$ and $h(\tx)$ is $C^2$ in $\tx$. We apply Lemma \ref{help} to $(g\cdot h)$, yielding
\begin{eqnarray}
\left|g(\tx)h(\tx)-\sum_{k=0}^{n+1}\lambda_k g(\tx_k)h(\tx_k)\right|
&\le&\max_{\tx\in \T_\nu}\|H(\tx)\|_2h_\nu^2\label{help3}
\end{eqnarray}
where the matrix $H(\tx)$ is defined by $(H(\tx))_{km}=\frac{\partial^2 (g\cdot h)(\tx)}{\partial x_k\partial x_m}$.
Note that
$\frac{\partial }{\partial x_m}(g\cdot h)=\frac{\partial g}{\partial x_m} h+g
\frac{\partial h}{\partial x_m}$ and
$$\frac{\partial^2 }{\partial x_k\partial x_m}(g\cdot h)=\frac{\partial^2 g}{\partial x_k\partial x_m} h+\frac{\partial g}{\partial x_m}\frac{\partial h}{\partial x_k}+\frac{\partial g}{\partial x_k}\frac{\partial h}{\partial x_m}+g\frac{\partial^2 h}{\partial x_k\partial x_m}.$$

Applying this to $g(\tx)= M_{il}(\tx)$, we observe that, since  $g(\tx)=M_{il}(\tx)$ is affine on the simplex, $\frac{\partial g}{\partial x_m}(\tx)=(w_{il}^\nu)_m$ and
$\frac{\partial^2 g}{\partial x_k\partial x_m} (\tx)=0$ for all $\tx\in \T_\nu$.
Hence,
$$\frac{\partial^2 }{\partial x_k\partial x_m}(g\cdot h)(\tx)=(w_{il}^\nu)_m\frac{\partial h(\tx)}{\partial x_k}+(w_{il}^\nu)_k\frac{\partial h(\tx)}{\partial x_m}+M_{il}(\tx)\frac{\partial^2 h(\tx)}{\partial x_k\partial x_m}.$$
Using $h(\tx)=(D_xf(\tx))_{lj}$, we obtain with
$\frac{\partial h}{\partial x_k}=\frac{\partial^2 f_l}{\partial x_k \partial x_j}$ and
$\frac{\partial^2 h}{\partial x_k\partial x_m}=\frac{\partial^3 f_l}{\partial x_k\partial x_m\partial x_j}$, $j\not=0$, defining $B_{3,\nu}= \max_{\tx\in{\T_\nu}, i,j,l\in \{0,\ldots,n\}}\left\|\frac{\partial^3 f(\tx)}{\partial x_i\partial x_j\partial x_l}\right\|_\infty$
\begin{eqnarray*}
|(H(\tx))_{km}|&=&\left|\frac{\partial^2 (g\cdot h)(\tx)}{\partial x_k\partial x_m}\right|\\
&\le&|(w^\nu_{il})_m|B_\nu+|(w^\nu_{il})_k|B_\nu +B_{3,\nu} |M_{il}(\tx)|
\end{eqnarray*}
Thus, using $\|H_1+H_2\|_2\le \|H_1\|_2+\|H_2\|_2$, as well as
$\|H_1\|_2\le \sqrt{n+1} \|H_1\|_1$,  $\|H_2\|_2\le \sqrt{n+1} \|H_2\|_\infty$ and $\|H\|_2\le (n+1)\|H\|_{max}$ we obtain
\begin{eqnarray*}\|H\|_2\
&\le&2\sqrt{n+1}\|w^\nu_{il}\|_1B_\nu+(n+1)B_{3,\nu}\max_{\tx\in \T_\nu}\max_{1\le i\le l\le n} |M_{il}(\tx)|\\
&\le&2\sqrt{n+1}D_\nu B_\nu+(n+1)B_{3,\nu}C_\nu,
\end{eqnarray*}
using Remark \ref{con_new} .
Hence, (\ref{help3}) establishes
\begin{eqnarray*}
\lefteqn{\left|\sum_{l=1}^nM_{il}(\tx)(D_xf(\tx))_{lj}-\sum_{l=1}^n\sum_{k=0}^{n+1}\lambda_k M_{il}(\tx_k)(D_xf(\tx_k))_{lj}\right|}\\
&\le&nh_\nu^2(2\sqrt{n+1}D_\nu B_\nu+(n+1)B_{3,\nu}C_\nu).
\end{eqnarray*}
which proves the lemma.
\end{proof}

Now we can estimate the value of $L_{M}$ for all points $\tx$ for the CPA metric $M$, given our constraints on the vertices.

\begin{lemma}\label{le2}
Let all constraints be satisfied. Then
the CPA metric $M$ defined in Definition \ref{CPA-def} fulfills:\begin{eqnarray}
\lambda_{max}(M(\tx)D_xf(\tx)+D_xf(\tx)^TM(\tx)+M'_+(\tx))
&\le&-1
\end{eqnarray}
for all $\tx\in \cD_K^\circ$.
\end{lemma}
\begin{proof}
The maximal eigenvalue is a convex and thus sublinear function, i.e. for $L,S\in\mathbb S^n$ symmetric (but not necessarily positive definite), we have
\begin{eqnarray}
\lambda_{max}(L+S)&\le& \lambda_{max}(L)+\lambda_{max}(S).\label{*}
\end{eqnarray}

We show that
 $|\lambda_{max} (S)|\le \|S\|_2$ holds for a symmetric (but not necessarily positive definite) matrix $S\in\mathbb R^{n\times n}$. Indeed, denote by $\lambda_1\le \ldots \le \lambda_n$ the eigenvalues of the symmetric matrix $S$ with corresponding eigenvectors $v_1,\ldots,v_n$, forming a basis of $\mathbb R^n$. Then
 $$Sv_i=\lambda_iv_i\Rightarrow S^TSv_i=S\lambda_iv_i=\lambda_i^2v_i.$$
 Thus, $v_i$ is an eigenvector of $S^TS$ with eigenvalue $\lambda_i^2$. Hence,
 \begin{eqnarray}\|S\|_2&=&\sqrt{\lambda_{max}(S^TS)}\nonumber\\
 &=&\max(|\lambda_1|,\ldots,|\lambda_n|)\nonumber\\
& =&\max(|\lambda_1|,|\lambda_n|)
 \ge |\lambda_n|=|\lambda_{max} (S)|.\label{2-eig}
 \end{eqnarray}

Fix $\tx\in\cD^\circ_K$ and $\T_\nu\in\cT^\cC_K$ such that $\tx+\theta\tf(\tx)\in\T_\nu$ for all $\theta\in[0,\theta^*]$ as in Definition \ref{basic}. Then $\tx=\sum_{k=0}^{n+1} \lambda_k\tx_k$, $\sum_{k=0}^{n+1}\lambda_k=1$ and $\lambda_k\in [0,1]$.
 Now, using (\ref{*}) and $\|S\|_2\le n\|S\|_{max}$, where $\|S\|_{max}=\max_{1\le i \le j \le n}|S_{ij}|$ for a matrix $S\in \mathbb S^n$, we have with Remark \ref{con_new}
\begin{eqnarray*}
\lefteqn{
\lambda_{max}(M(\tx)D_xf(\tx)+D_xf(\tx)^TM(\tx)+M'_+(\tx))}\\
&\le&
\lambda_{max}\left(
\sum_{k=0}^{n+1}\lambda_k\left(M(\tx_k)D_xf(\tx_k)+D_xf(\tx_k)^TM(\tx_k)+(w^\nu_{ij}\cdot \tf(\tx_k))_{i,j=1,\ldots,n}\right)\right)
\\
&&+\lambda_{max}\bigg(
M(\tx)D_xf(\tx)+D_xf(\tx)^TM(\tx)+M'_+(\tx)\\
&&\hspace{1.5cm}-
\sum_{k=0}^{n+1}\lambda_k\left(M(\tx_k)D_xf(\tx_k)+D_xf(\tx_k)^TM(\tx_k)+(w^\nu_{ij}\cdot \tf(\tx_k))_{i,j=1,\ldots,n}\right)\bigg)\\
&\le&\sum_{k=0}^{n+1}\lambda_k\lambda_{max}\left(
M(\tx_k)D_xf(\tx_k)+D_xf(\tx_k)^TM(\tx_k)+(w^\nu_{ij}\cdot \tf(\tx_k))_{i,j=1,\ldots,n}\right)
\\
&&+\bigg\|M(\tx)D_xf(\tx)+D_xf(\tx)^TM(\tx)+M'_+(\tx)\\
&&\hspace{1.5cm}-
\sum_{k=0}^{n+1}\lambda_k\left(M(\tx_k)D_xf(\tx_k)+D_xf(\tx_k)^TM(\tx_k)+(w^\nu_{ij}\cdot \tf(\tx_k))_{i,j=1,\ldots,n}\right)
\bigg\|_2\\
&\le&\sum_{k=0}^{n+1}\lambda_k(-1-E_\nu)
+n\bigg\|M(\tx)D_xf(\tx)+D_xf(\tx)^TM(\tx)+M'_+(\tx)\\
&&
\hspace{1.5cm}-
\sum_{k=0}^{n+1}\lambda_k\left(M(\tx_k)D_xf(\tx_k)+D_xf(\tx_k)^TM(\tx_k)+(w^\nu_{ij}\cdot \tf(\tx_k))_{i,j=1,\ldots,n}\right)\bigg
\|_{max}\\
&\le&-1-E_\nu
+E_\nu\\
&=&-1
\end{eqnarray*}
by Lemma \ref{4.12} and the definition of $E_\nu$.
\end{proof}

We summarize the results of this section in the following theorem.

\begin{theorem}\label{summary}
Let $\cT^\cC_K$ be a triangulation as in Definition \ref{triconstr} and $\cD_K:=\bigcup_{\T_\nu\in\cT^\cC_K}\T_\nu$.
Let all constraints be satisfied. Then
the CPA metric defined in Definition \ref{CPA-def} fulfills:
\begin{enumerate}
\item $M(\tx)$ is symmetric, positive definite and a $T$-periodic function and thus defines a Riemannian metric on $\cD_K^\circ$ in the sense of Definition \ref{Riem}.
\item $L_M(\tx)\le -\frac{1}{2\mu_{max}}\le -\frac{1}{2C}<0$ for all $\tx \in \cD_K^\circ$, where $\mu_{max}:=\max_{\tx\in\cD_K}\lambda_{max}(M(\tx))>0$, $L_M$ is defined in Theorem \ref{th1} and $C:=\max_{\T_\nu\in\cT^\cC_K}C_\nu$.
    \item If $G\subset \cD_K^\circ$ is positively invariant, then
    $$\int_0^T p(t)^T M'_+(t,x(t))p(t)\, dt$$ exists and is finite for all solutions
$x(t)$ with $x(0)\in G$ and all functions $p\in C^0([0,T], \mathbb R^n)$.
\end{enumerate}
Hence, $M$ satisfies all assumptions of Theorem \ref{th1}.
\end{theorem}
\begin{proof}
Part 1. follows directly from Lemma \ref{le1} and \ref{le4.9}.

For Part 2., note that $\mu_{max}$ exists and is positive, since $M$  is positive definite for all $\tx$ and depends continuously on $\tx$. Then
\begin{eqnarray*}
L_M(\tx)&=&
\sup_{w\in \mathbb R^n\setminus \{0\}} \frac{w^T\left[M(\tx)D_xf(\tx)
+\frac{1}{2}M'_+(\tx)\right]w}{w^T M(\tx)w}\\
&\le&\frac{1}{2}\frac{\lambda_{max}(M(\tx)D_xf(\tx)+D_xf(\tx)^TM(\tx)
+M'_+(\tx))}{\lambda_{max}(M(\tx))}\\
&\le&-\frac{1}{2\mu_{max}}
\end{eqnarray*}
using Lemma
\ref{le2}. Moreover, Constraint 2. yields that $\mu_{max}\le \max_{\T_\nu \in\cT^\cC_K} C_\nu\le  C$.

Part 3. follows from Lemma \ref{le4.9} using that $G\subset \cD_K^\circ$ is positively invariant. This shows the theorem.
\end{proof}

\subsection{Objective function}\label{objective}

While we are primarily interested in the calculation of any Riemannian contraction metric, i.e. a feasible solution of the semidefinite optimization problem, the optimization problem also allows for an objective function.
Theorem \ref{summary} suggests a possible objective function, namely simply the maximum over all $C_\nu$, i.e. $C=\max_{\T_\nu\in\cT^\cC_K}C_\nu$. We can either implement this by adding $C$ as an additional variable with constraints $$C_\nu\le C$$ for all $\nu$. Or we can replace the variables $C_\nu$ in Constraint 2. by the uniform bound $C$ directly, which reduces the number of Constraints 2. to $v$, the number of vertices. By minimizing the constant $C$, we minimize the bound $-\frac{1}{2C}$ on the maximal Floquet exponent. Note, however, that $C$ has a lower bound given by $\epsilon_0$. This means that by choosing $\epsilon_0$ too large, the estimate on the Floquet exponent will be very rough.

\section{Feasibility of the semidefinite optimization problem}
\label{sec3}


In the next theorem we assume that there exists an exponentially stable periodic orbit. Then we can show that the semidefinite optimization problem has a feasible  solution and we can thus construct a suitable Riemannian metric. We have to assume that the triangulation is fine enough, i.e. in practice we start with a coarse triangulation and refine until we obtain a solution. The triangulation has to stay suitably regular, i.e. the angles in simplices have lower and upper bounds. For simplicity, we use the reference simplicial complex and scale it uniformly as described in Section \ref{tria}, but other refinements are also possible.

\begin{theorem}\label{feas}
Let the system $\dot{x}=f(t,x)$, $f\in C^2(S^1_T\times \R^n,\R^n) $ have an exponentially stable periodic orbit $\Omega$ with basin of attraction $A(\Omega)$. Let $\cC\subset A(\Omega)$, $\cC\in \cN$, $\Omega\subset \cC$ be a compact set in the cylinder $S^1_T \times \mathbb R^n$.
Fix $\epsilon_0>0$.

Then there is a $K^*\in\mathbb N$ such that the semidefinite optimization problem is feasible for all triangulations $\cT^\cC_K$  as described in Section \ref{tria} with $K\ge K^*$.
Note that we can choose the constants $C_\nu$ and $D_\nu$ to be the same for each simplex, i.e. $C_\nu=C$ and $D_\nu=D$ for all $\nu$.
\end{theorem}
\begin{proof}
 \begin{enumerate}

\item \textbf{Smooth Riemannian metric}

Denote the  maximal real part of the Floquet exponents of the exponentially stable periodic orbit $\Omega$
 by $-\nu_0<0$ and set $\epsilon:=\frac{\nu_0}{2}$.

Since $\cC\subset A(\Omega)$ is compact and $A(\Omega)$ open, there is a positive Euclidean distance between $\cC$ and the boundary of $A(\Omega)$. Let $d>0$ denote this distance if it is finite and otherwise set $d:=1$.
Now define $\cC^*$ to be the set of all $\tx\in S^1_T\times \R^n$ that have Euclidean distance less than or equal to $d/2$ to $\cC$.  Clearly $\cC^*\in\cN$.
For all large enough $K\in\N_0$ the Euclidean distance from the boundary of  $\displaystyle \cD_K:= \bigcup_{\T_\nu\in\cT^\cC_K}\T_\nu$
to $\cC$ is bounded above by $\max_{\T_\nu\in\cT^\cC_K}
\diam(\T_\nu)\le S^* 2^{-K}T$, see (\ref{hnu}). Thus, there is a $K^{**}\in\N_0$ such that $\cC \subset \cD_K \subset \cC^*$ for every
$K \ge K^{**}$.

Now apply the following theorem to the above defined $\cC^*$ and $\epsilon$.

\begin{theorem}[Theorem 4.2, \cite{zaa}] \label{lyapweight}
Assume that $f\in C^0(\mathbb R\times
\mathbb R^n,\mathbb R^n)$ is a periodic function in $t$ with period $T$,
and all partial derivatives of order one with respect to $x$ are
continuous functions of $(t,x)$.
Let $\Omega:=\{(t,\tilde{x}(t))\in  S^1_T\times \mathbb R^n\}$
 be an exponentially asymptotically stable
periodic orbit, $A(\Omega)$ be its basin of attraction,
and let the maximal real part of the Floquet exponents
 be $-\nu_0<0$.
Then for all $\epsilon>0$ and all compact sets $\cC^*$ with
 $\Omega\subset  \cC^*\subset A(\Omega)$
 there exists a  Riemannian
metric $\widetilde{M}\in C^1(\cC^*,\mathbb R^{n\times n})$,
such that ${ L}_{\widetilde{M}}(t,x)\le -\nu_0+\epsilon$ holds for all  $(t,x)\in \cC^*$.
\end{theorem}
\begin{remark}
The proof of this theorem, cf. \cite[Theorem 4.2]{zaa} shows that we have $\widetilde{M}\in C^2(\cC^*,\mathbb R^{n\times n})$, if $f\in C^2(S^1_T\times \mathbb R^n,\mathbb R^n)$, as is the case by our assumptions.
\end{remark}

 Since both $\widetilde{M}(\tx)$ and $\lambda_{min}$  are continuous functions and $\widetilde{M}$ is positive definite for all $\tx\in\cC^*$, there exists an $\epsilon_1>0$ such that
$$\lambda_{min}(\widetilde{M}(\tilde{x}))\ge \epsilon_1$$
for all $\tx\in\cC^*$, as $\cC^*$ is compact.

Moreover, since $L_{\widetilde{M}}$ is a continuous function (note that $\widetilde{M}$ is smooth) satisfying $L_{\widetilde{M}}(\tx)\le -\frac{\nu_0}{2}$ for all $\tx\in\cC^*$, we have
\begin{eqnarray*}
\lambda_{max}(\widetilde{M}(\tx)D_xf(\tx)+D_xf(\tx)^T\widetilde{M}(\tx)+\widetilde{M}'(\tx))&<&0
\end{eqnarray*}
for all $\tx\in\cC^*$. As $\widetilde{M}(\tx)D_xf(\tx)+D_xf(\tx)^T\widetilde{M}(\tx)+\widetilde{M}'(\tx)$ and  $\lambda_{max}$  are continuous functions, there  exists an $\epsilon_2>0$ such that
 $$\lambda_{max}(\widetilde{M}(\tx)D_xf(\tx)+D_xf(\tx)^T\widetilde{M}(\tx)+\widetilde{M}'(\tx))\le -\epsilon_2$$
for all $\tx\in\cC^*$, as $\cC^*$ is compact.

 Now define $M(\tx):=\max\left(\frac{\epsilon_0}{\epsilon_1},\frac{2}{\epsilon_2}\right)\widetilde{M}(\tx)$.
 Then
 \begin{eqnarray}
\lambda_{min}(M(\widetilde{x}))&\ge& \epsilon_0\label{cond4}\\
\lambda_{max}(M(\tx)D_xf(\tx)+D_xf(\tx)^TM(\tx)+M'(\tx))&\le& -2\label{cond5}
\end{eqnarray}
 for all $\tx\in \cC^*$.

$M$ is $C^2$ on the compact set $\cC^*$, so that we can define the following constants
\begin{eqnarray*}
M_0^*&:=&\max_{\tx\in \cC^*}\lambda_{max}(M(\tx))\\
M_1^*&:=&\max_{\tx\in \cC^*}\max_{1\le i\le j\le n}\|\nabla_{\tx} M_{ij}(\tx)\|_1\\
M_2^*&:=&\max_{\tx\in \cC^*}\max_{1\le i\le j\le n}\|H_{M_{ij}}(\tx)\|_2
\end{eqnarray*}
where $H_{M_{ij}}(\tx)$ denotes the Hessian of $M_{ij}(\tx)$.

\item \textbf{Assigning the variables of the optimization problem}

 Set
 \begin{eqnarray*}  F&:=&\max_{\tx\in \cC^*}\|\tf(\tx)\|_\infty\\
 B&:=&\max_{\tx\in \cC^*, i,j\in \{0,\ldots,n\}}\left\|\frac{\partial^2 f(\tx)}{\partial x_i\partial x_j}\right\|_\infty\\
B_{3}&:=&\max_{\tx\in \cC^*, i,j,k\in \{0,\ldots,n\}}\left\|\frac{\partial^3 f(\tx)}{\partial x_i\partial x_j\partial x_k}\right\|_\infty\mbox{ if }f\in C^3.
 \end{eqnarray*}

Finally, we define, using the constant $X^*$ from Lemma \ref{Xlemma},
   \begin{eqnarray*}
C^*&:=&M_2^* (n+1)\left(\frac{S^*X^*}{2s^*} + \sqrt{n+1}\right)\\
C&:=&M_0^*\\
  D&:=& (C^*+M_1^*)(n+1)\\
  h_1^*&:=&\left\{\begin{array}{ll}\,[2nB(\sqrt{n+1}D+2n(n+1)C]^{-1}&\mbox{  if }f\in C^2\\
\,[2n(\sqrt{n+1}(1+4n)BD+2n(n+1)B_3C)]^{-1/2}&\mbox{  if }f\in C^3\end{array}\right.\\
h_2^*&:=&\frac{1}{2nFC^*}.
  \end{eqnarray*}
%

%


 Now define
 $$K^*:=\max\left(\left\lceil \frac{\ln (S^* T)-\ln (\min(h^*_1,h^*_2,1))}{\ln 2}\right\rceil,K^{**}\right),$$
 let $K\ge K^*$
  and consider the triangulation $\cT^{\cC}_K$. Note that, since $K\ge K^*$, we have by (\ref{hnu})
  \begin{eqnarray}h_\nu\le S^*2^{-K}T&\le& S^*2^{-K^*} T\le \min(h_1^*,h_2^*,1).\label{hnu2}
   \end{eqnarray}
   Interpolate $M(\tx)$ on this triangulation, i.e.  assign the variables $M_{ij}(\tx_k)$ for all vertices $\tx_k$ of the triangulation; moreover, set $C_\nu=C$ and $D_\nu=D$ for all $\nu$. By (\ref{hnu2})   and the definition of $h_1^*$  this ensures
 \begin{eqnarray}
 E_\nu&\le&\frac{1}{2}.\label{enu}
 \end{eqnarray}

\item \textbf{Auxiliary results}

  To check the feasibility in the next step, we need some estimates.

Consider a simplex $\T_{\nu}=\co(\tx_0,\tx_1,\ldots,\tx_{n+1})\in\cT^{\cC}_K$. Denote, as in Constraint 3., (\ref{defw})
$w^\nu_{ij} :=
X_{K,\nu}^{-1} M_{ij,\nu}^*$, where $X_{K,\nu}=\left(\begin{array}{c}(\tx_1-\tx_0)^T\\
(\tx_2-\tx_0)^T\\
\vdots\\
(\tx_{n+1}-\tx_0)^T
\end{array}\right)$ and
\begin{eqnarray}
M_{ij,\nu}^*:=
\begin{pmatrix}
    M_{ij}(\btx_1)-M_{ij}(\btx_0) \\
    M_{ij}(\btx_2)-M_{ij}(\btx_0) \\
    \vdots\\
    M_{ij}(\btx_{n+1})-M_{ij}(\btx_0)
  \end{pmatrix}.\label{wknu}
  \end{eqnarray}
  Note that $M_{ij}(\tx)$ is two times continuously differentiable on $\T_{\nu}\subset \cC^*$ and for
$i,j\in \{1,2,\ldots,n+1\}$ we thus have by Taylor's theorem
$$
M_{ij}(\btx_k)=M_{ij}(\btx_0)+\nabla_{\tx} M_{ij}(\btx_0)\cdot (\btx_k - \btx_0) + \frac{1}{2}(\btx_k-\btx_0)^T H_{M_{ij}}(\btz_k)(\btx_k-\btx_0),
$$
where $H_{M_{ij}}$ is the Hessian of $M_{ij}$ and $\btz_k=\btx_0+\vartheta_k(\btx_k-\btx_0)$ for some $\vartheta_k \in\,]0,1[$.

By rearranging terms and combining, this delivers with (\ref{wknu})
\begin{equation}
\label{wmXnW}
M_{ij,\nu}^* - X_{K,\nu} \nabla_{\tx} M_{ij}(\btx_0)=\frac{1}{2}
\begin{pmatrix}
    (\btx_1-\btx_0)^T H_{M_{ij}}(\btz_1)(\btx_1-\btx_0)\\
    (\btx_2-\btx_0)^T H_{M_{ij}}(\btz_2)(\btx_2-\btx_0)\\
    \vdots\\
    (\btx_{n+1}-\btx_0)^T H_{M_{ij}}(\btz_{n+1})(\btx_{n+1}-\btx_0)
  \end{pmatrix}.
\end{equation}
We have
$$
\left|(\btx_k-\btx_0)^T H_{M_{ij}}(\btz_k)(\btx_k-\btx_0)\right| \leq h_\nu^2 \|H_{M_{ij}}(\btz_k)\|_2 \leq M_2^* h_\nu^2.
$$
Note that $\btz_k\in \T_\nu$ since the simplex is convex.
Hence, by (\ref{wmXnW}),
\begin{equation}
\|M_{ij,\nu}^* - X_{K,\nu} \nabla_{\tx} M_{ij}(\btx_0)\|_1
  \leq \frac{n+1}{2}M_2^* h_\nu^2. \label{Ain1}
\end{equation}

Now we need to obtain an estimate on $\nabla_{\tx} M_{ij}(\btx_k)$.
For $k\in\{1,2,\ldots,n+1\}$, $l\in\{0,\ldots,n\}$, there is a $\btz_{kl}$ on the line segment between $\btx_k$ and $\btx_0$, such that
$$
\partial_l M_{ij}(\btx_k)-\partial_lM_{ij}(\btx_0)=\nabla_{\tx}
\partial_lM_{ij}(\btz_{kl})\cdot(\btx_k-\btx_0),
$$
where $\partial_lM_{ij}$ denotes the $l$-th component of $\nabla_{\tx} M_{ij}$ and $\nabla_{\tx}\partial_l M_{ij}$ is the gradient of this function.
Then  we have
$$
|\partial_lM_{ij}(\btx_k)-\partial_lM_{ij}(\btx_0)|\leq \|\nabla_{\tx}\partial_lM_{ij}(\btz_{kl})\|_2\|\btx_k-\btx_0\|_2\leq \sqrt{n+1}M_2^*h_\nu.
$$
Hence,
\begin{align*}
\|\nabla_{\tx} M_{ij}(\btx_k)- \nabla_{\tx} M_{ij}(\btx_0)\|_1 \leq (n+1)^{3/2}M_2^*h_\nu.
\end{align*}
From this, Lemma \ref{Xlemma} and (\ref{Ain1}) we obtain the estimate
\begin{eqnarray}
\lefteqn{\|w^\nu_{ij} -\nabla_{\tx} M_{ij}(\btx_k)\|_1}\n\\
&=&
\| X_{K,\nu}^{-1}M_{ij,\nu}^* - \nabla_{\tx} M_{ij}(\btx_k)\|_1\n\\
&\le&
\| X_{K,\nu}^{-1}M_{ij,\nu}^* - \nabla_{\tx} M_{ij}(\btx_0)\|_1
+ \|\nabla_{\tx} M_{ij}(\btx_0) - \nabla_{\tx} M_{ij}(\btx_k)\|_1\n\\
&\le&
\| X_{K,\nu}^{-1}\|_1 \|M_{ij,\nu}^* - X_{K,\nu} \nabla_{\tx} M_{ij}(\btx_0)\|_1
+ (n+1)^{3/2}M_2^*h_\nu\n\\
&\le&\frac{2^KX^*}{s^* T}
   \frac{n+1}{2}M_2^* h_\nu^2+ (n+1)^{3/2}M_2^*h_\nu\n\\
&\le&
\frac{S^*X^*}{s^*}   \frac{n+1}{2}M_2^* h_\nu+ (n+1)^{3/2}M_2^*h_\nu\n\\
   &=&C^*h_\nu,\label{xfirst}
\end{eqnarray}
using (\ref{hnu}) and the definition of
 $C^*$.

A further useful consequence, which we need later, is that
\begin{eqnarray}\| X_{K,\nu}^{-1}M_{ij,\nu}^* \|_1
&\le&
\| X_{K,\nu}^{-1}M_{ij,\nu}^* - \nabla_{\tx} M_{ij}(\tx_k)\|_1
+\|  \nabla_{\tx} M_{ij}(\tx_k)\|_1\n\\
&\le&  C^*  h_\nu+M_1^*.
\label{cfirst}
\end{eqnarray}

\item \textbf{Feasibility}

   Now we check that the assignment in Step 2.   is a feasible point of the semidefinite optimization problem. We discuss each constraint in the following.

   \begin{enumerate}[{Constraint} 1.]
   \item This follows directly from the periodicity of $M$.
   \item This follows directly from the definition of $C$.
   \item
   We have for all $1\le i\le j \le n$ and any $l=0,\ldots,n$ by (\ref{cfirst})
$$
|(w_{ij}^\nu)_l|\le
\|w_{ij}^\nu\|_1=
\| X_{K,\nu}^{-1}M_{ij,\nu}^* \|_1
\le    C^*h_\nu+M_1^*\le \frac{D}{n+1}
$$
by definition of $D$
as $h_\nu\le 1$.
   \item This follows from  (\ref{cond4}).
   \item This is the main step and will be shown below.
   \end{enumerate}
%
%
%
%

%

We fix a simplex $\T_\nu=\co(\tx_0,\ldots,\tx_{n+1})\in \cT^{\cC}_K$ and a vertex $\tx_k$, $k\in \{0,\ldots,n+1\}$.
With Remark \ref{con_new}
we need to show that
 $$\lambda_{max} \left(M(\tx_k)D_xf(\tx_k)+D_xf(\tx_k)^TM(\tx_k)+(w_{ij}^\nu \cdot \tf(\tx_k))_{i,j=1,\ldots,n}\right)+E_{\nu}\le -1.$$
We have $E_\nu\le \frac{1}{2}$, cf. (\ref{enu}). Hence, using the sublinearity of $\lambda_{max}$, we have
 \begin{eqnarray*}
 \lefteqn{\hspace{-1.3cm}\
 \lambda_{max} \left(M(\tx_k)D_xf(\tx_k)+D_xf(\tx_k)^TM(\tx_k)+(w_{ij}^\nu \cdot \tf(\tx_k))_{i,j=1,\ldots,n}\right)+E_{\nu}}\\
 &\le&
 \lambda_{max} \left(M(\tx_k)D_xf(\tx_k)+D_xf(\tx_k)^TM(\tx_k)+M'(\tx_k)\right)\\
 &&
 +\lambda_{max}((w_{ij}^\nu \cdot \tf(\tx_k))_{i,j=1,\ldots,n}-M'(\tx_k))+\frac{1}{2}\\
 &\le&-2+\frac{1}{2} +n \max_{1\le i\le j\le n}|w_{ij}^\nu \cdot \tf(\tx_k)-M_{ij}'(\tx_k)|
  \end{eqnarray*}
 where we have used (\ref{cond5}) and $\lambda_{max}(S)\le \|S\|_2\le n\|S\|_{max}$, cf. (\ref{2-eig}).

 Using $M_{ij}'(\tx_k)=\nabla_{\tx} M_{ij}(\tx_k)\cdot \tf(\tx_k)$, we obtain
 \begin{eqnarray*}
 \lefteqn{\hspace{-0.9cm}
 \lambda_{max} \left(M(\tx_k)D_xf(\tx_k)+D_xf(\tx_k)^TM(\tx_k)+(w_{ij}^\nu \cdot \tf(\tx_k))_{i,j=1,\ldots,n}\right)+E_{\nu}}\\
 &\le&-\frac{3}{2}+n \max_{1\le i\le j\le n}\|w_{ij}^\nu -\nabla_{\tx} M_{ij}(\tx_k)\|_1
 \cdot \|\tf(\tx_k)\|_\infty\\
 &\le&-\frac{3}{2}+nC^* h_\nu F\\
 &\le&-\frac{3}{2}+\frac{1}{2}=-1
 \end{eqnarray*}
 with (\ref{xfirst}), the definition of $F$ and using $h_\nu\le h^*_2=\frac{1}{2nFC^*}$.
\end{enumerate}
 This proves that the constraints are fulfilled and the optimization problem has a feasible solution.
\end{proof}

%


\noindent
{\bf Acknowledgement} The first author would like to thank Michael Hinterm\"uller and Thomas Surowiec for helpful discussions and comments regarding optimization.

\end{document}